\documentclass{article}

\usepackage{graphicx}
\usepackage{color}
\usepackage{indentfirst}
\usepackage{amsmath,amsfonts,amsthm,amssymb}
\usepackage{mathrsfs}
\usepackage{amscd}
\usepackage{hyperref}

\def\co{\colon\thinspace}
\DeclareMathAlphabet{\mathsfsl}{OT1}{cmss}{m}{sl}

\newtheorem{thm}{Theorem}[section]
\newtheorem{lem}[thm]{Lemma}
\newtheorem{cor}[thm]{Corollary}
\newtheorem{prop}[thm]{Proposition}
\newtheorem{conj}[thm]{Conjecture}
\newtheorem*{thm*}{Theorem}

\theoremstyle{definition}

\newtheorem{rem}[thm]{Remark}
\newtheorem{condition}[thm]{Condition}

\newtheorem{known facts}[thm]{Known Facts}

\begin{document}

\def\G{{\Gamma}}
  \def\d{{\delta}}
  \def\ci{{\circ}}
  \def\e{{\varepsilon}}
  \def\l{{\lambda}}
  \def\L{{\Lambda}}
  \def\m{{\mu}}
  \def\n{{\nu}}
  \def\o{{\omega}}
  \def\O{{\Omega}}
  \def\Th{{\Theta}}\def\s{{\sigma}}
  \def\v{{\varphi}}
  \def\a{{\alpha}}
  \def\b{{\beta}}
  \def\p{{\partial}}
  \def\r{{\rho}}
  \def\ra{{\rightarrow}}
  \def\lra{{\longrightarrow}}
  \def\g{{\gamma}}
  \def\D{{\Delta}}
  \def\La{{\Leftarrow}}
  \def\Ra{{\Rightarrow}}
  \def\x{{\xi}}
  \def\c{{\mathbb C}}
  \def\z{{\mathbb Z}}
  \def\2{{\mathbb Z_2}}
  \def\q{{\mathbb Q}}
  \def\t{{\tau}}
  \def\u{{\upsilon}}
  \def\th{{\theta}}
  \def\la{{\leftarrow}}
  \def\lla{{\longleftarrow}}
  \def\da{{\downarrow}}
  \def\ua{{\uparrow}}
  \def\nwa{{\nwtarrow}}
  \def\swa{{\swarrow}}
  \def\nea{{\netarrow}}
  \def\sea{{\searrow}}
  \def\hla{{\hookleftarrow}}
  \def\hra{{\hookrightarrow}}
  \def\sl{{SL(2,\mathbb C)}}
  \def\ps{{PSL(2,\mathbb C)}}
  \def\qed{{\hspace{2mm}{\small $\diamondsuit$}\goodbreak}}
  \def\pf{{\noindent{\bf Proof.\hspace{2mm}}}}
  \def\ni{{\noindent}}
  \def\sm{{{\mbox{\tiny M}}}}
   \def\sf{{{\mbox{\tiny F}}}}
   \def\sc{{{\mbox{\tiny C}}}}
  \def\ke{{\mbox{ker}(H_1(\partial M;\2)\ra H_1(M;\2))}}
  \def\et{{\mbox{\hspace{1.5mm}}}}
 \def\sk{{{\mbox{\tiny K}}}}
\def\sxz{{{\mbox{\tiny $X_1$}}}}
\def\sxo{{{\mbox{\tiny $X_0$}}}}
\def\sxi{{{\mbox{\tiny $X_i$}}}}
\def\sxt{{{\mbox{\tiny $X_2$}}}}
\numberwithin{equation}{section}

\title{Finite Dehn  surgeries on knots in $S^3$}

\author{{Yi NI}\\{\normalsize Department of Mathematics, Caltech}\\
{\normalsize 1200 E California Blvd, Pasadena, CA 91125}
\\{\small\it Email\/:\quad\rm yini@caltech.edu}
\\\\
{Xingru ZHANG}
\\
{\normalsize Department of Mathematics,
University at Buffalo}\\
{\small\it Email\/:\quad\rm xinzhang@buffalo.edu}}

\date{}
\maketitle

\begin{abstract}
We show that on  a hyperbolic knot $K$ in $S^3$,  the distance between any two finite  surgery slopes is at most two and consequently there are at most  three nontrivial finite surgeries. Moreover in case that $K$  admits three nontrivial finite surgeries,  $K$ must be the  pretzel knot $P(-2,3,7)$. In case that $K$ admits two noncyclic finite surgeries or two finite surgeries at distance two, the two surgery slopes must be one of ten  or seventeen  specific  pairs respectively. For $D$-type finite surgeries, we improve a finiteness theorem due to Doig by giving an explicit bound on the possible resulting  prism manifolds, and also prove that $4m$ and $4m+4$ are characterizing slopes for the torus knot $T(2m+1,2)$ for each $m\geq 1$.
\end{abstract}

\section{Introduction}
A Dehn surgery on a knot $K$ in $S^3$ with slope $p/q$ (which is parameterized by the standard meridian/longitude coordinates of $K$)
is called a  cyclic or finite surgery
if the resulting manifold, which we denote by $S^3_K(p/q)$, has cyclic or finite fundamental group respectively.  By \cite{Ga}  $S^3_K(0)$ has cyclic fundamental  group
only when $K$ is the trivial knot.
It follows that cyclic surgery on nontrivial knots in $S^3$ is equivalent to finite cyclic surgery.
Due to Perelman's resolution of Thurston's  Geometrization Conjecture, a connected closed  3-manifold has  finite fundamental group
if and only if it is a spherical space form.
For a spherical space form $Y$, it has cyclic fundamental group if and only if
 it is a lens space, and it has non-cyclic fundamental group if and only if it has a Seifert fibred structure
whose base orbifold is $S^2(a,b,c)$, a $2$-sphere with
three cone points of orders $a\leq b\leq c$, satisfying $\frac{1}{a}+
\frac{1}{b}+ \frac{1}{c}>1$, i.e.
$(a,b,c)=(2,3,3)$, or $(2,3,4)$, or $(2,3,5)$, or $(2,2,n)$ for some integer
  $n>1$.
Correspondingly we say that a spherical space form $Y$ or its fundamental group
is of $C$-type, or  $T$-type,  $O$-type, $I$-type, $D$-type   if $Y$ is a lens space,
or a Seifert fibred space with base orbifold  $S^2(2,3,3)$, $S^2(2,3,4)$, $S^2(2,3,5)$, $S^2(2,2,n)$ respectively.
 We shall also refine finite surgeries (slopes) into
 $C$-type (often called cyclic), $T$-type, $O$-type, $I$-type and $D$-type  accordingly.

If a non-hyperbolic knot in $S^3$ admits a nontrivial finite surgery, then
the knot is either a torus knot or a cable over a torus knot \cite{BZ1},
and finite surgeries on torus knots and  on cables over torus knots
are classified in \cite{Moser} and \cite{BH} respectively.
Concerning finite surgeries on  hyperbolic knots in $S^3$,  we recall the following
\begin{known facts}\label{known1} Let $K\subset S^3$ be any fixed hyperbolic knot.
\newline
(1) Any nontrivial cyclic surgery slope of $K$ must be an integer,
 $K$ has at most two nontrivial cyclic
  surgery slopes, and if two, they are consecutive integers \cite{CGLS}.
  \newline
  (2) Any  finite surgery slope of $K$ must be either
  an integer or a half integer \cite{BZ1},  the distance between any two
finite surgery slopes of $K$ is at most three, $K$ has at most four nontrivial finite  surgery slopes \cite{BZ5}, and consequently  $K$ has at most one
half-integer finite surgery slope.
\newline
(3) The distance between a finite surgery slope and a cyclic surgery slope on $K$ is at most two
\cite{BZ1}.\newline
(4) Any $D$-type finite surgery slope of $K$ must be an integer.
There is at most one $D$-type finite surgery slope on $K$.
If there is a $D$-type finite surgery on $K$,  then there is at most one nontrivial cyclic surgery
on $K$, and the $D$-type finite surgery slope and the cyclic surgery slope  are consecutive
integers \cite{BZ1}.
\newline
(5) Any $O$-type finite surgery slope of $K$ must be an integer.
If there is an $O$-type finite surgery on $K$,  then there is at most one nontrivial cyclic surgery
on $K$, and the $O$-type finite surgery slope and the cyclic surgery slope  are consecutive
integers \cite{BZ1}.
\newline
(6) There are at most two $T$-type finite surgery slopes on $K$ and if two, one is integral, the
other is half-integral, and their distance is three \cite{BZ1}.
\end{known facts}

The above results were obtained by using  `classical' techniques,
 mostly those derived from  $(P)SL_2(\c)$-representations
of $3$-manifold groups, hyperbolic geometry  and
 geometric and combinatorial topology in dimension $3$.
Recently new progresses on finite surgeries on knots in $S^3$ have been made
through applications of Floer holomogy theory. Note that up to replacing
a knot $K$ by its mirror image $-K$, we may and shall assume that any finite surgery slope
on a nontrivial knot in $S^3$ has positive sign.
We recall the following

\begin{known facts}\label{known2}
(1) If a knot $K$ in $S^3$ admits a nontrivial finite surgery, then $K$ is a fibred knot \cite{NiFibred}.
\newline
(2) If a nontrivial  knot $K$ in $S^3$ admits a nontrivial  finite surgery slope $p/q$, then $\frac{p}{q}\geq 2g(K)-1$
where $g(K)$ is the Seifert genus of $K$ \cite{OSzRatSurg}, and the nonzero coefficients
of the Alexander polynomial of $K$ are alternating $\pm 1$'s \cite{OSzLspace}.
\newline
(3) If a knot $K$ in $S^3$
admits a $T$-type or $O$-type or $I$-type finite surgery with an integer  slope, then the surgery slope
is one of the finitely many  integers listed
in Tables 1--3  in Section~\ref{sect:TOI type surgeries} and there is
a sample knot $K_0$ (also listed in these tables) on which the same surgery slope yields the same spherical space form, and $K$ and $K_0$ have  the same knot Floer homology  \cite{Gu}.
\newline
(4) If a knot $K$ in $S^3$
admits a $T$-type  or $I$-type finite surgery with
a  half-integer  slope, then the surgery slope
is one of the ten slopes listed
in Table~\ref{table:Half} (with the two on trefoil knot omitted as they
can only be
realized on trefoil knot) in Section~\ref{sect:TOI type surgeries} and there is
a sample knot $K_0$ (also listed in the table) on which the same surgery slope yields the same spherical space form, and $K$ and $K_0$ have  the same knot Floer homology \cite{LiNi}.
\newline
(5) If a knot $K$ in $S^3$
admits an integer $D$-type finite surgery  slope $p\leq 32$, then $p$
is one of the slopes listed in Table~\ref{table:D} in Section ~\ref{sect:TOI type surgeries}
and there is a sample knot $K_0$ (also listed in the table) on which the same
surgery slope yields the same $D$-type spherical space form, and $K$ and $K_0$ have  the same knot Floer homology  \cite{D1}.
\newline
(6)  If a knot  $K$ in $S^3$ admits a cyclic surgery with an integer slope $p$,
 then there is a Berge knot $K_0$ (given in \cite{Berge}) such that $S^3_K(p)=S^3_{K_0}(p)$,
and $K$ and $K_0$ have the same knot Floer homology \cite{Greene}.
 \newline
 (7) If $p$ is a cyclic surgery slope for a hyperbolic knot $K$ in $S^3$, then $p=14$ or $p\geq 18$. Moreover if $p\geq 4g(K)-1$, then $K$ is a Berge knot  \cite{Baker}.
\end{known facts}

The purpose of this paper is to update and improve
results on finite surgeries on hyperbolic knots
 in $S^3$, applying various techniques and results
 combined together.

As recalled in Know Facts~\ref{known2} (4) that
there are only ten  specific half-integer  slopes, listed in Table~\ref{table:Half},
each of which could possibly be a finite surgery slope for some hyperbolic knot
in $S^3$, our first result  exclude two of them.

\begin{thm}\label{17/2and23/2}
 Each of $17/2$ and  $23/2$ can never be a finite surgery slope for a hyperbolic knot in $S^3$.
 \end{thm}

A  main result of this paper is the following

\begin{thm}\label{main thm}
Let $K$ be any fixed hyperbolic knot in $S^3$.
\newline
(1) The distance between any two
finite  surgery slopes on $K$ is at most two. Consequently
 there are at most  three  nontrivial finite surgeries on $K$.
 \newline
  (2) If  $K$ admits
three nontrivial finite surgeries, then
$K$ must be the pretzel knot $P(-2,3,7)$.
\newline
(3) If $K$ has two non-cyclic finite surgeries,  the surgery slopes
are one of the following ten
pairs, the knot $K$ has the same knot Floer homology as
the sample knot $K_0$ given along the pair and
the pair of slopes yield the same spherical space forms  on $K_0$:
$$\begin{array}{llll}\{43/2,21, [11,2;3,2]\},& \{53/2,27, [13,2;3,2]\},&
\{103/2,52, [17,3;3,2]\},&
\{113/2,56, [19,3;3,2]\},\\
\{22,23, P(-2,3,9)\},&
\{28,29, -K(1,1,0)\},&
\{50,52, [17,3;3,2]\},&
\{56,58, [19,3;3,2]\},\\
\{91, 93, [23,4;3,2]\},&
\{99,101, [25,4;3,2]\}.\end{array}$$
 \newline
(4)  If  $K$  admits
two finite surgery slopes which are distance two apart, then the two  slopes
are one of the following seventeen
pairs, the knot $K$ has  the same knot Floer homology as
the sample knot $K_0$ given along the pair and
the pair of slopes yield the same spherical space forms on $K_0$:
$$\begin{array}{llll}\{43/2,1/0, [11,2;3,2]\}, &\{45/2,1/0, [11,2;3,2]\},&
\{51/2,1/0, [13,2;3,2]\},& \{53/2,1/0, [13,2;3,2]\},\\
\{77/2,1/0, [19,2;5,2]\},&\{83/2,1/0, [21,2;5,2]\},&\{103/2,1/0, [17,3;3,2]\},&
\{113/2,1/0, [19,3;3,2]\},\\ \{17,19, P(-2,3,7)\},&
\{21,23,[11,2;3,2]\},& \{27,25,[13,2;3,2]\},& \{37,39,[19,2;5,2]\},\\
 \{43,41,[21,2;5,2]\}&\{50,52, [17,3;3,2]\},
&\{56,58, [19,3;3,2]\},&
\{91, 93, [23,4;3,2]\},\\
\{99,101, [25,4;3,2]\}.&&&\end{array}$$
 \end{thm}

The notations for the sample knots will be explained in Section~\ref{sect:TOI type surgeries}.
Note that a sample knot is not necessarily hyperbolic.
When a sample knot is non-hyperbolic, the corresponding case of the described finite surgeries
on a hyperbolic knot possibly never happen.

Parts of the theorem are sharp; on the pretzel knot $P(-2,3,7)$ (which is hyperbolic),
$17$, $18$, $19$ are three finite surgery slopes, $17$ being $I$-type and $18, 19$ being cyclic,
 and on the pretzel knot  $P(-2,3,9)$ (which is also hyperbolic),
 $22$ and $23$ are two non-cyclic finite surgery slopes,
 $22$ being  $O$-type and $23$ being $I$-type.

A $D$-type spherical space form is also called a prism manifold.
Let $P(n,m)$ be the prism manifold with Seifert invariants
\[(-1;(2,1),(2,1),(n,m)),\]
where the base orbifold has genus 0, $n>1$, $\gcd(n,m)=1$.
%(The underlying topological space of the base is a sphere, and $P(n,m)$ is oriented. For simplicity, we omit the base genus and the orientation type information in the Seifert data.)
Every prism manifold can be expressed in this
form. As a byproduct of the proof of Theorem~\ref{main thm}, we have the following

\begin{thm}\label{prism and det}
If $P(n,m)$ can be obtained by Dehn surgery on a knot $K$ in $S^3$, then
$n<|4m|$. 
\end{thm}

Theorem~\ref{prism and det} improves \cite[Theorem~2]{D2} where no explicit bound on $n$ was given. (In an earlier preprint of \cite{D2}, the author claimed
a bound of $n<|16m|$ without proof.)

Recall that on a torus knot $T(2m+1,2)$, $4m$ and $4m+4$  are $D$-type
finite surgery slopes.
Our next main result implies that 
 each of the prism manifolds $S^3_{T(2m+1,2)}(4m)$ and
$S^3_{T(2m+1,2)}(4m+4)$ can not be obtained by surgery on any other knot in $S^3$
besides $\pm T(2m+1,2)$.

\begin{thm}\label{thm:DiSurg}
Suppose that $S^3_K(4n)\cong\varepsilon S^3_{T(2m+1,2)}(4n)$ for some
 $\varepsilon\in\{\pm\}$ and $n=m$ or $m+1$, where $\e\in\{\pm \}$ stands for an orientation.
Then $\varepsilon=+$ and $K=T(2m+1,2)$.
\end{thm}

In the terminology of \cite{NiZhang}, the above theorem implies that $4m$ and $4m+4$ are characterizing slopes for $T(2m+1,2)$, that is, whenever $S^3_K(4n)\cong S^3_{T(2m+1,2)}(4n)$ for $n=m$ or $m+1$,
then $K=T(2m+1,2)$.

Combining Theorem \ref{thm:DiSurg} with  Known Facts \ref{known2} (5) and Known Facts \ref{known1} (4), we 
have 

\begin{cor}Any $D$-type finite surgery slope of a hyperbolic knot in $S^3$  is an integer larger than or equal to
$28$. 
\end{cor}

The bound $28$ can be realized as a $D$-type slope on two hyperbolic knots in $S^3$
(see Table~\ref{table:D} in Section~\ref{sect:TOI type surgeries}).

The results described above suggest the following updated
conjectural picture concerning  finite surgeries on hyperbolic knots in $S^3$.

\begin{conj}Let $K$ be a hyperbolic knot in $S^3$.\newline
(1) {\rm (The Berge conjecture) } If $K$ admits a nontrivial cyclic surgery, then $K$ is a primitive/primitive
knot as defined in \cite{Berge} (i.e. a Berge knot).
\newline
(2) {\rm (Raised in \cite{Gu})} If $K$ admits a $T$-type or $O$-type or $I$-type finite surgery, then $K$ is one of the twenty-three hyperbolic
sample knots listed in Tables 1--3.
\newline
(3) $K$ does not have any half-integral finite surgery slope.
\newline
(4) If $K$ admits two non-cyclic finite surgeries, then $K$ is either  $P(-2,3,9)$ or
 $-K(1,1,0)$.
\newline
(5) If $K$ admits two  finite surgeries at distance two, then $K$ is $P(-2,3,7)$.
\newline
(6) If $K$ admits a non-cyclic finite surgery, then $K$ is a primitive/Seifert-fibered knot.
\newline
(7) {\rm (Improved from \cite[Conjecture~12]{D2})} If the prism manifold $P(n,m)$ can be obtained by surgery on $K$, then $n<2|m|-2$.
\end{conj}

The proofs of the above theorems are mainly using $PSL_2(\c)$ representation techniques, the correction terms from Heegaard Floer homology and the Casson--Walker invariant
besides Known Facts~\ref{known1} and~\ref{known2}.
In Section~\ref{sect:TOI type surgeries} we give a bit detailed explanation
 about Known Facts~\ref{known2} (3) (4) (5),  which will be convenient to be applied in later sections.
In Section~\ref{sect:CS norm}, we recall briefly some machinery
of using  $PSL_2(\c)$-representations for studying finite surgeries,
specialized to the case for hyperbolic knots in $S^3$.
We prove Theorem~\ref{17/2and23/2} in Section~\ref{sect:17/2and23/2}
 where an outline of proof will be indicated at the beginning.
 The method of proof is mainly $PSL_2(\c)$-representation techniques,
 combined with Known Facts~\ref{known2} (4)  as well as various other results.
 We then present the proof of Theorem~\ref{main thm},  which we split into two
 parts, corresponding to the two cases whether
  a half-integer finite surgery slope is involved or not.
 Part I of the proof, given in Section~\ref{sect:main thm-part I},   is mainly using $PSL_2(\c)$-representation techniques
combined with Known Facts~\ref{known2} as well as various other results,
and part II of the proof, given in  Section~\ref{sect:main thm-part II}, is mainly using the Casson-Walker invariant combined with Known Facts~\ref{known2}.
 Section~\ref{sect:main thm-part II} also contains the proof of Theorem~\ref{prism and det}.
In Section~\ref{sect:DiSurg}, we prove Theorem~\ref{thm:DiSurg} applying Heegaard Floer homology and some
topological arguments.

\vspace{5pt}\noindent{\bf Acknowledgements.}\quad  The first author was
partially supported by NSF grant numbers DMS-1103976, DMS-1252992, and an Alfred P. Sloan Research Fellowship.

%%%%%
%%%%%
%%%%%
%%%%%
%%%%%

\section{Tables of finite surgeries}\label{sect:TOI type surgeries}

Given a rational homology sphere $Y$ and a Spin$^c$ structure $\mathfrak s\in\mathrm{Spin}^c(Y)$, Ozsv\'ath and Szab\'o defined a rational number $d(Y,\mathfrak s)$ called the {\it correction term} \cite{OSzAbGr}.
Recall that a rational homology sphere $Y$ is an {\it L-space} if its Heegaard Floer homology $\widehat{HF}(Y,\mathfrak s)$ is isomorphic to $\mathbb Z$ for each $\mathfrak s\in\mathrm{Spin}^c(Y)$. Lens spaces, and more generally spherical $3$--manifolds are L-spaces. The correction terms are the only informative Heegaard Floer invariants for L-spaces.

Let $L(p,q)$ be the lens space obtained by $\frac{p}q$--surgery on the unknot in $S^3$, and we fix a particular identification $\mathrm{Spin}^c(L(p,q))\cong \mathbb Z/p\mathbb Z$.
The correction terms for lens spaces can be computed inductively as in \cite{OSzAbGr}:
\begin{eqnarray}
d(S^3,0)&=&0,\nonumber\\
d(L(p,q),i)&=&-\frac14+\frac{(2i+1-p-q)^2}{4pq}-d(L(q,r),j),\label{eq:CorrRecurs}
\end{eqnarray}
where $0\le i<p+q$,
$r$ and $j$ are the reductions modulo $q$ of $p$ and $i$, respectively.

For a knot $K$ in $S^3$, suppose its Alexander polynomial normalized by the conditions $\Delta_K(t)=\Delta_K(t^{-1})$ and $\Delta_K(1)=1$ is
$$\Delta_K(t)=\sum_ia_it^i.$$
Define a sequence of integers
$$t_i=\sum_{j=1}^{\infty}ja_{i+j},\quad i\ge0.$$
If $K$ admits a L-space surgery, then one can prove \cite{OSzRatSurg,RasThesis}
\begin{equation}\label{eq:tsProperties}
t_i\ge0,\quad t_i\ge t_{i+1}\ge t_i-1,\quad t_{g(K)}=0,
\end{equation}
and  the correction terms of $S^3_{K}(p/q)$ can be computed in terms of $t_i$'s by the formula
\[%\begin{equation}\label{eq:CorrSurg}
d(S^3_{K}(p/q),i)=d(L(p,q),i)-2t_{\min\{\lfloor\frac{i}q\rfloor,\lfloor\frac{p+q-i-1}q\rfloor\}}.
\]%\end{equation}

The above results give us a necessary condition for a spherical space form $Y$ with $H_1(Y)\cong\mathbb Z/p\mathbb Z$ to be the $p/q$--surgery on any knot $K\subset S^3$. That is
\begin{condition}\label{cond:Corr}
There exist a sequence of integers $\{t_i\}_{i\ge0}$ satisfying
\[
t_i\ge0,\quad t_i\ge t_{i+1}\ge t_i-1,\quad t_{i}=0 \text{ when }i\gg0,
\]
and
a symmetric affine isomorphism $\phi\co \mathbb Z/p\mathbb Z\to \mathbb Z/p\mathbb Z$ such that
\[
d(Y,\phi(i))=d(L(p,q),i)-2t_{\min\{\lfloor\frac{i}q\rfloor,\lfloor\frac{p+q-i-1}q\rfloor\}}.
\]
\end{condition}
Here $\phi$ is symmetric means that $\phi$ commutes with the conjugation of Spin$^c$ structures once we identify $\mathbb Z/p\mathbb Z$ with the corresponding sets of Spin$^c$ structures.

Condition~\ref{cond:Corr} is easy to check using a simple computer program.
When Condition~\ref{cond:Corr} is satisfied, we can recover the Alexander polynomial of the possible knot $K$ admitting the surgery to $Y$.  It is surprising that
Condition~\ref{cond:Corr} is also sufficient in all known cases. For example, Ozsv\'ath and Szab\'o \cite[Proposition~1.13]{OSzLspace} confirmed that a lens space $L(p,q)$ with $p\le1500$ can be obtained by $p$--surgery on a knot in $S^3$ if and only if Condition~\ref{cond:Corr} holds. Doig \cite{D1} proved similar results for prism spaces with $|H_1|\le 32$, and
Gu \cite{Gu} and Li--Ni \cite{LiNi} proved similar results for $T$-type, $O$-type, $I$-type spherical space forms.

Here we list the $T$-type, $O$-type, $I$-type and some  $D$-type finite surgery slopes
and the sample knots mentioned in Known Facts~\ref{known2} (3) (4) (5),
which are reproduced from \cite{Gu,LiNi,D1} for
the reader's convenience.
For each of these  finite surgeries, we actually   list   four to six  relevant data: $p$ (or $p/2$), $Y$, $K_0$, $g$, $\det(K)$, $\Delta_K''(1)$,
which are useful in this and  later sections. Here $p>0$
 (or $p/2>0$) is the finite surgery slope, $Y$ is the resulting manifold, $K_0$ is a sample knot which admits the finite surgery, $g$,
  $\det(K)$ and $\Delta_K(t)$ are respectively the Seifert genus, the determinant and
  the normalized Alexander polynomial
  for any knot $K\subset S^3$ which admits the finite surgery.

It is not hard to see that every $T$-, $O$- or $I$-type spherical space form, up to orientation reversing, can be obtained by Dehn filling on $\mathbb T$, the exterior of the right-hand trefoil knot in $S^3$. Thus we represent the corresponding surgery manifold  $Y$ by Dehn filling on $\mathbb T$ and specify the orientation.

Now we explain the notations we use for the sample knots in these tables.
Let $-K$ be the mirror image of $K$.
\begin{itemize}
\item Many of the  knots in the tables are torus knots or iterated torus knots. Following \cite{BH}, we use $[p,q;r,s]$ to denote the $(p,q)$--cable of the torus knot $T(r,s)$.
\item There are two hyperbolic pretzel knots in the tables: $P(-2,3,7)$ and $P(-2,3,9)$.
\item Following \cite{MM}, let $K(p,q,r,n)$ be the twist torus knot obtained by applying $n$ full twists to $r$ parallel strings in $T(p,q)$, where $p,q$ are coprime integers, $q>|p|\ge2$, $0\le r\le p+q$, $n\in\mathbb Z$. It is proved in \cite{MM}  that
the $pq+n(p+q)^2$--surgery on $K(p,q,p+q,n)$ yields a Seifert fibered manifold with base orbifold $S^2(|p|,q,|n|)$.
\item Let $B(p,q;a)$ be the Berge knot \cite{Berge} whose dual is the simple knot \cite{RasBerge} in the homology class $a$ in $L(p,q)$.
\item Three knots $K_2^{\#},K_3^*$, $K_3$  are from \cite[Section 10]{BZ1}.
\item One knot $K(1,1,0)$ is from \cite[Section~4]{EM} which is also the knot $K_1$ given in \cite[Proposition 18]{BH}.
\item Three knots are primitive/Seifert  knots from \cite{BergeKang}, we will use the notation there, starting with ``P/SF$_d$ KIST''.
\end{itemize}

\begin{table}
\centering
\caption{Integral $T$-type surgeries}\label{table:T}
\begin{tabular}{|c|c|c|c|c|c|}
\hline
$p$ &$Y$ &$K_0$ &$g$ &$\det(K)$ &$\Delta_K''(1)$ \\
\hline
3 &$\mathbb T(3)$ &$T(3,2)$ &1 &3 &2 \\
\hline
9 &$\mathbb T(9)$ &$T(3,2)$ &1 &3 &2 \\
\hline
21 &$\mathbb T(21/4)$ &$[11, 2; 3, 2]$ &7 &11 &38 \\
\hline
27 &$\mathbb T(27/4)$ &$[13, 2; 3, 2]$ &8 &13 &50 \\
\hline
51 &$-\mathbb T(51/8)$ &$K_2^{\#}$  &18 &1 &200\\
\hline
69 &$-\mathbb T(69/11)$ &$-K(2, 3, 5, -3)$ &25 &1 &368\\
\hline
81 &$\mathbb T(81/13)$ &$K(2, 3, 5,  3)$ &31 &1 &536\\
\hline
93 &$\mathbb T(93/16)$ &$[23, 4; 3, 2]$ &37 &23 &692\\
\hline
99 &$\mathbb T(99/16)$ &$[25, 4; 3, 2]$  &40 &25 &812\\
\hline
\end{tabular}
\end{table}

\begin{table}
\centering
\caption{Integral $O$-type surgeries}\label{table:O}
\begin{tabular}{|c|c|c|c|c|c|c|c|}
\hline
$p$ &$Y$ &$K_0$ &$g$ &$\det(K)$ &$\Delta_K''(1)$\\
\hline
2 &$\mathbb T(2)$ &$T(3,2)$ &1 &3 &2\\
\hline
10 &$\mathbb T(10)$ &$T(3,2)$ &1  &3 &2\\
\hline
10 &$-\mathbb T(10)$ &$T(4,3)$ &3  &3 &10\\
\hline
14 &$-\mathbb T(14/3)$ &$T(4,3)$ &3 &3 &10 \\
\hline
22 &$\mathbb T(22/3)$ &$P(-2,3,9)$ &6 &3 &34 \\
\hline
38 &$\mathbb T(38/7)$ &$B(39, 16; 16)$ &13 &3 &114 \\
\hline
46 &$-\mathbb T(46/7)$ &$B(45, 19; 8)$ &16 &3 &162 \\
\hline
50 &$\mathbb T(50/9)$ &$[17, 3; 3, 2]$ &19 &3 &210 \\
\hline
58 &$\mathbb T(58/9)$ &$[19,3;3,2]$  &21 &3 &258 \\
\hline
62 &$-\mathbb T(62/11)$ &P/SF$_d$ KIST III $(-5,-3,-2,-1,1)$ &23 &3 &306 \\
\hline
70 &$\mathbb T(70/11)$ &$B(71, 27; 11)$ &27 &3 &402\\
\hline
86 &$\mathbb T(86/15)$ &$-K(3, 4, 7, -2)$ &33 &3 &586\\
\hline
94 &$-\mathbb T(94/15)$ &$-K(2, 3, 5, -4)$ &35 &3 &690\\
\hline
106 &$\mathbb T(106/17)$ &$K(2, 3, 5, 4)$ &41 &3 &914\\
\hline
106 &$-\mathbb T(106/17)$ &$[35,3;4,3]$ &43 &3 &906\\
\hline
110 &$-\mathbb T(110/19)$ &$K(3, 4, 7, 2)$ &45 &3 &1002\\
\hline
110 &$-\mathbb T(110/19)$ &$[37,3;4,3]$ &45 &3 &1002\\
\hline
146 &$\mathbb T(146/25)$ &$[29,5;3,2]$ &61 &3 &1730\\
\hline
154 &$\mathbb T(154/25)$ &$[31,5;3,2]$ &65 &3 &1970\\
\hline
\end{tabular}
\end{table}

\begin{table}
\centering
\caption{Integral $I$-type surgeries}\label{table:I}
\begin{tabular}[t]{|c|c|c|c|c|c|}
\hline
$p$ &$Y$ &$K_0$ &$g$ &$\det(K)$ &$\Delta_K''(1)$  \\
\hline
1 &$\mathbb T(1)$ &$T(3,2)$ &1 &3 & 2 \\
\hline
7 &$\mathbb T(7/2)$ &$T(5,2)$ &2 &5 &6 \\
\hline
11 &$\mathbb T(11)$ &$T(3,2)$ &1 &3 & 2 \\
\hline
13 &$-\mathbb T(13/3)$ &$T(5,2)$ &2 &5 &6\\
\hline
13 &$\mathbb T(13/3)$ &$T(5,3)$ &4 &1 &16 \\
\hline
17 &$\mathbb T(17/2)$ &$T(5,3)$ &4 &1 &16\\
\hline
17 &$-\mathbb T(17/2)$ &$P(-2,3,7)$ &5 &1 &24 \\
\hline
19 &$\mathbb T(19/4)$ &$[9,2;3,2]$ &6  &9 &28\\
\hline
23 &$\mathbb T(23/3)$ &$P(-2,3,9)$ &6 &3 &34 \\
\hline
29 &$\mathbb T(29/4)$ &$[15,2;3,2]$ &9 &15 &64 \\
\hline
29 &$-\mathbb T(29/4)$ &$-K(1,1,0)$ &9 &11 &62 \\
\hline
37 &$-\mathbb T(37/7)$ &$-K(-2,5,3,-3)$ &11 &1 &96 \\
\hline
37 &$\mathbb T(37/7)$ &$[19,2;5,2]$ &13 &19 &114 \\
\hline
43 &$-\mathbb T(43/8)$ &$[21,2;5,2]$ &14 &21 &134 \\
\hline
47 &$\mathbb T(47/7)$ &$K_3^*$  &16 &1 &168 \\
\hline
49 &$\mathbb T(49/9)$ &$[16,3;3,2]$ &18 &9 &188 \\
\hline
59 &$\mathbb T(59/9)$ &$[20,3;3,2]$ &22 &9 &284\\
\hline
83 &$-\mathbb T(83/13)$ &P/SF$_d$ KIST V$(1,-2,-1,2,2)$  &32 &1 &552\\
\hline
91 &$\mathbb T(91/16)$ &$[23,4;3,2]$ &37 &23 &692\\
\hline
101 &$\mathbb T(101/16)$ &$[25,4;3,2]$ &40 &25 &812\\
\hline
113 &$-\mathbb T(113/18)$ &$-K(3,5,8,-2)$ &45 &1 &1024\\
\hline
113 &$\mathbb T(113/18)$ &$-$P/SF$_d$ KIST V$(-3,-2,-1,2,2)$ &46 &1 &1048\\
\hline
119 &$-\mathbb T(119/19)$ &$-K(2,3,5,-5)$  &45 &1 &1112\\
\hline
131 &$\mathbb T(131/21)$ &$K(2,3,5,5)$ &51 &1 &1392\\
\hline
133 &$\mathbb T(133/23)$ &$[44,3;5,3]$ &55 &3 &1434\\
\hline
137 &$-\mathbb T(137/22)$ &$-K(2,5,7,-3)$ &55 &3 &1506\\
\hline
137 &$\mathbb T(137/22)$ &$[46,3;5,3]$ &57 &3  &1554\\
\hline
143 &$\mathbb T(143/23)$ &$K(3,5,8,2)$ &60 &1 &1696\\
\hline
157 &$\mathbb T(157/27)$ &$[39,4;5,2]$ &65 &39 &1996\\
\hline
157 &$-\mathbb T(157/27)$ &$K(2,5,7,3)$  &65 &3 &2034\\
\hline
163 &$-\mathbb T(163/28)$ &$[41,4;5,2]$ &68 &41 &2196\\
\hline
211 &$\mathbb T(211/36)$ &$[35,6;3,2]$ &91 &35 &3642\\
\hline
221 &$\mathbb T(221/36)$ &$[37,6;3,2]$ &96 &37 &4062\\
\hline
\end{tabular}
\end{table}

\begin{table}
\centering
\caption{Half-integral surgeries}\label{table:Half}
\begin{tabular}{|c|c|c|c||c|c|c|c|}
\hline
$p/2$ & $Y$ & $K_0$ &$g$ &$p/2$ & $Y$ & $K$ &$g$\\
\hline
${17}/2$ &$-\mathbb T({17}/2)$ &$T(5,2)$ &2 & ${53}/2$ &$\mathbb T({53}/8)$ &$[13,2;3,2]$ &8\\
\hline
 ${23}/2$ &$\mathbb T({23}/3)$ &$T(5,2)$ &2 &${77}/2$ &$-\mathbb T({77}/{12})$ &$[19,2;5,2]$ &13\\
\hline
 ${43}/2$ &$\mathbb T({43}/8)$ &$[11,2;3,2]$ &7 &${83}/2$ &$\mathbb T({83}/{13})$ &$[21,2;5,2]$ &14\\
\hline
 ${45}/2$ &$\mathbb T({45}/8)$ &$[11,2;3,2]$ &7 &${103}/2$ &$\mathbb T({103}/{18})$ &$[17,3;3,2]$ &19\\
\hline
 ${51}/2$ &$\mathbb T({51}/8)$ &$[13,2;3,2]$ &8 &${113}/2$ &$\mathbb T({113}/{18})$ &$[19,3;3,2]$ &21\\
\hline
\end{tabular}
\end{table}

\begin{table}
\centering
\caption{Integral $D$-type $p$-surgeries with $p\leq 32$}\label{table:D}
\begin{tabular}{|c|c|c|c|c|}
\hline
$p$ &$K_0$ &$g$ &$\det(K)$ &$\Delta_K''(1)$ \\
\hline
4 &$T(3,2)$ &1 &3 &2 \\
\hline
8 &$T(3,2)$ &1 &3 &2 \\
\hline
8&$T(5,2)$ &2 &5 &6 \\
\hline
12&$T(5,2)$ &2 &5 &6 \\
\hline
12&$T(7,2)$ &3 &7 &12 \\
\hline
16 &$T(7,2)$ &3 &7 &12\\
\hline
16 &$T(9,2)$ &4 &9 &20\\
\hline
20 &$T(9,2)$ &4 &9 &20\\
\hline
20 &$T(11,2)$ &5 &11 &30\\
\hline
24 &$T(11,2)$ &5 &11 &30\\
\hline
24 &$T(13,2)$ &6 &13 &42\\
\hline
28 &$T(13,2)$ &6 &13 &42\\
\hline
28 &$-K(1,1,0)$ &9 &11 &62\\
\hline
28 &$K_3$ &8 &5 &54\\
\hline
28 &$T(15,2)$ &7 &15 &56\\
\hline
32&$T(15,2)$ &7 &15 &56\\
\hline
32&$T(17,2)$ &8 &17 &72\\\hline
\end{tabular}
\end{table}

\begin{rem}
In \cite{BH}, the authors enumerated all finite surgeries on iterated torus knots. However, one case was missed in their list: the $58$--surgery on $[19,3;3,2]$ yields the $O$-type manifold $\mathbb T(58/9)$. This mistake was inherited in \cite{Gu}, where the author found a knot on which the $58$--surgery yields $\mathbb T(58/9)$, but she thought the knot was hyperbolic  because
this case was not listed in \cite{BH}.
\end{rem}

In Table~\ref{table:Half} we omit the two half-integer finite surgeries on $T(3,2)$
as it was already known that only $T(3,2)$ can have such surgeries \cite{OSz3141}.

\begin{lem}\label{different A-poly}
In Tables~\ref{table:T}--\ref{table:D}, any two sample knots expressed in different notations
are different knots with different Alexander polynomials.
\end{lem}
\begin{proof}[{\bf Proof}]
By Known Facts~\ref{known2} (1), all the knots in $S^3$ with finite surgeries are fibred. So if two
sample knots in the tables have different genera, they must have different Alexander polynomials. Below we will compare the Alexander polynomials of sample knots with the same genera. Since $\det(K)=|\Delta_K(-1)|$, we often just compare $\det(K)$ and $\Delta_K''(1)$.

\begin{itemize}
\item$g=3$. There are two knots $T(4,3)$ and $T(7,2)$. They have different $\det(K)$.

\item$g=4$. There are two knots $T(5,3)$ and $T(9,2)$. They have different $\det(K)$.

\item$g=5$. There are two knots $P(-2,3,7)$ and $T(11,2)$. They have different $\det(K)$.

\item$g=6$. There are three knots $[9,2;3,2]$, $P(-2,3,9)$ and $T(13,2)$. They have different $\det(K)$.

\item$g=7$. There are two knots $[11,2;3,2]$ and $T(15,2)$. They have different $\det(K)$.

\item$g=8$. There are three knots $[13,2;3,2]$, $T(17,2)$ and $K_3$. They have different $\det(K)$.

\item$g=9$. There are two knots $[15,2;3,2]$ and $-K(1,1,0)$ with different $\det(K)$.

\item$g=13$. There are two knots $B(39,16;16)$ and $[19,2;5,2]$ with different $\det(K)$.

\item$g=16$. There are two knots $B(45,19;8)$ and $K_3^*$ with different $\det(K)$.

\item$g=18$. There are two knots $K_2^{\#}$ and $[16,3;3,2]$ with different $\det(K)$.

\item$g=45$. There are four knots. The two knots with $O$-type surgery can be distinguished from the two with $I$-type surgery via $\det(K)$. The two knots with $I$-type surgery have different $\Delta_K''(1)$. The two knots with $O$-type surgery have different
    Alexander polynomials as given below:
\begin{eqnarray*}
\Delta_{K(3, 4, 7, 2)}(t)&=&1-(t^2+t^{-2})+(t^{3}+t^{-3)}-(t^{5}+t^{-5})+(t^{7}+t^{-7})+\cdots,\\
\Delta_{[37, 3; 4, 3]}(t)&=&1-(t^2+t^{-2})+(t^{3}+t^{-3)}-(t^{5}+t^{-5})+(t^{6}+t^{-6})+\cdots.
\end{eqnarray*}

\item$g=55$. There are two knots $[44, 3; 5, 3]$ and $K(2, 5, 7, -3)$ with different $\Delta_K''(1)$.

\item$g=65$. There are three knots $[31, 5; 3, 2]$, $[39, 4; 5, 2]$ and $K(2, 5, 7, 3)$ with different $\Delta_K''(1)$.
\end{itemize}

This finishes the proof.
\end{proof}

\begin{rem}In Tables (1)-(5), any sample knot expressed in a notation different from that
of a torus knot or a cable over torus knot is a hyperbolic knot.
This is because that those torus knots and cables over torus knots
appeared in Tables (1)-(5) constitute the set of all non-hyperbolic knots in $S^3$ which admit $T$-type or $O$-type or $I$-type finite surgeries or
 $D$-type integer $p$-surgeries with $p\leq 32$ and by Lemma \ref{different A-poly} all other sample knots
are different from these non-hyperbolic ones.
\end{rem}

%%%%%
%%%%%
%%%%%
%%%%%
%%%%%

\section{$PSL_2(\c)$-character variety, Culler-Shalen norm or semi-norm,
and finite surgery}\label{sect:CS norm}

In this section we briefly review some machinery  and results from \cite{CGLS,BZ1,BZ2,BZ3} used in studying  cyclic and finite surgeries,
but specialized to the case of  knots in $S^3$. In fact for simplicity we shall
mainly  restrict our discussion
to the following  very special situation:  any hyperbolic knot in $S^3$ which is assumed to have
a half integer finite surgery. This is sufficient for our purpose in this paper.

 For a finitely generated group $\G$ we use $R(\G)$ to denote the
  $PSL_2(\c)$-representation variety of $\G$. (The term  variety used here means complex affine algebraic set). Let $\Phi: SL_2(\c)\ra PSL_2(\c)$ be the canonical quotient homomorphism.
  Given  an element  $\Upsilon$ in $PSL_2(\c)$,
 $\Phi^{-1}(\Upsilon)=\{A, -A\}$ for some $A\in SL_2(\c)$ and we
 often simply write $\Upsilon=\pm A$.
 In particular we may define
$tr^2(\Upsilon):=[trace(A)]^2$, which is obviously  well defined on
   $\Upsilon$.
   An element   $\Upsilon\in PSL_2(\c)$ is said to be parabolic
   if it is not the identity element $\pm I$ and satisfies $tr^2(\Upsilon)=4$.

A representation $\r\in R(\G)$ is said to be irreducible if
it is not conjugate to a representation whose image lies
in $$\{\pm \left(\begin{array}{cc}
a&b\\
0&a^{-1}\end{array}\right);\;\; a,b \in \c, a\ne 0\}.$$
 A representation $\r\in R(M)$ is said to be strictly irreducible if
it is irreducible and is not conjugate to a representation whose image lies
in $$\{\pm \left(\begin{array}{cc}
a&0\\
0&a^{-1}\end{array}\right), \pm \left(\begin{array}{cc}
0&b\\
-b^{-1}&0\end{array}\right);\;\; a,b \in \c, a\ne 0, b\ne 0\}.$$

 For a representation $\r\in R(\G)$, its character $\chi_\r$ is  the function
   $\chi_\r:\G\ra \c$ defined by
   $\chi_\r(\g)=tr^2(\r(\g))$ for each $\g\in \G$.
 Let $X(\G)=\{\chi_\r; \r\in R(\G)\}$ denote the set of characters of representations of $\G$.
 Then $X(\G)$ is also a complex affine algebraic set, usually referred as the $PSL_2(\c)$-character variety of
 $\G$.

A character $\chi_\r\in X(\G)$  is said to be irreducible or strictly irreducible or discrete faithful
or dihedral  if the  representation $\r$ has the corresponding property.

 Let $t:R(\G)\ra X(\G)$ denote the natural onto map defined by $t(\r)=\chi_\r$.
 Then $t$ is a regular map between the two algebraic sets.
 For an element $\g\in \G$, the function
$f_\g:X(\G)\ra \c$ is defined by $f_\g(\chi_\r)=\chi_\r(\g)-4=tr^2(\r(\g))-4$ for each $\chi_\r\in X(\G)$.
Each  $f_\g$ is a regular function on $X(\G)$.
Obviously $\chi_\r\in X(\G)$ is a zero point of $f_\g$ if and only if either
$\r(\g)=\pm I$ or $\r(\g)$ is a parabolic element. It is also evident that $f_\g$ is invariant when $\g$ is replaced by  a conjugate of $\g$ or by the inverse of $\g$.

If $\phi:\G\ra \G'$ is a surjective homomorphism between two finitely generated groups,
 it naturally induces an embedding
of $R(\G')$ into  $R(\G)$ and an embedding
of $X(\G')$ into $X(\G)$. So we may simply
consider $R(\G')$ and $X(\G')$ as subsets of $R(\G)$ and $X(\G)$ respectively,
and write $R(\G')\subset R(\G)$ and $X(\G')\subset X(\G)$.

 For a connected compact manifold $Y$, let $R(Y)$ and $X(Y)$ denote $R(\pi_1(Y))$
 and $X(\pi_1(Y))$ respectively.

Let $M$ be the  exterior of a  knot $K$ in $S^3$. A slope on $\p M$ is called a boundary slope
if there is an orientable properly embedded incompressible and boundary-incompressible
surface $F$ in $M$ whose boundary $\p F$ is a non-empty set of parallel
essential curves in $\p M$ of slope $\g$.
For a slope $\g$ on $\p M$, $M(\g)$ denotes the Dehn filling  of $M$ with slope $\g$.
Throughout we let $\m$ denote the meridian slope and $\l$ the canonical longitude slope on
$\p M$.  We use $\D(\g_1,\g_2)$ to denote the distance between two slopes $\g_1$ and $\g_2$ on $\p M$.
We call $K$ or $M$ hyperbolic if
the interior of $M$
supports a complete hyperbolic metric of finite volume.
Note that for any slope $\g$, there is a surjective homomorphism
from $\pi_1(M)$ to $\pi_1(M(\g))$, and thus $R(M(\g))\subset R(M)$ and $X(M(\g))\subset X(M)$.

For the exterior $M$ of a nontrivial knot in $S^3$, $H_1(\p M;\z)\cong \pi_1(\p M)$  can be considered as a subgroup of $\pi_1(M)$ which is well defined up to conjugation. Hence  the function $f_\g$ on $X(M)$ is well defined for each class $\g\in H_1(\p M;\z)$.
 As $\g$  is also invariant under the change of the orientation of $\g$, $f_\g$ is also well defined when $\g$ is a slope in $\p M$.
For convenience we will often  not make a distinction
 among  a primitive class of $H_1(\p M;\z)$, the corresponding
  element of $\pi_1(\p M)$ and the corresponding slope in $\p M$, that is, we shall often use these terms exchangeably under the same notation.

\begin{lem}\label{characters of finite groups}{\rm (\cite[Lemma 5.3]{BZ1})}
Let $Y$ be  a spherical space form.
\newline(1) If $Y$ is of $T$-type, then $X(Y)$ has exactly  one irreducible character $\chi_\r$ and the image of $\r$ is isomorphic to
the tetrahedral group $T_{12}=\{x, y; x^2 = y^3=(xy)^3 =1\}$.
\newline(2) If $Y$ is of $O$-type, then $X(Y)$  has exactly two
irreducible characters $\chi_{\r_0}$, $\chi_{\r_1}$
and one of $\r_0$ and $\r_1$ has its image isomorphic to
the octahedral group $O_{24}=\{x, y; x^2 = y^3=(xy)^4 =1\}$ (we name this character the
$O$-type character of $X(Y)$) and the other
has image isomorphic to the dihedral group $D_6=\{x,y; x^2=y^2=(xy)^3=1\}$ (we name this character
the $D$-type character of $X(Y)$).
\newline(3) If $Y$ is of $I$-type,  then $X(Y)$  has exactly two
irreducible characters $\chi_{\r_1}$, $\chi_{\r_2}$
and both $\r_1$ and $\r_2$ have image isomorphic to
the icosahedral group $I_{60}=\{x, y; x^2 = y^3=(xy)^5 =1\}$.
\end{lem}

Let $M$ be the exterior of a knot in $S^3$
and suppose that $\b$ is a $D$-type or $T$-type or $O$-type or $I$-type
finite surgery slope on $\p M$.
Let $\rho\in R(M(\b))\subset R(M)$ be an irreducible representation and  we require
 $\rho$ to have image $O_{24}$ when $\b$ is of $O$-type.
 Let $\phi$ denote the composition
$\pi_1(\partial M) \to \pi_1(M) \to \pi_1(M(\b)) \stackrel{
\rho}{\to} \ps$ and let $q = | \phi(\pi_1(\partial M))|$.
Then $q$ is uniquely associated to the finite surgery slope $\b$
and following \cite{BZ5}, we  say more specifically that  the finite surgery slope
$\b$ is of type $D(q)$, $T(q)$, $O(q)$ or $I(q)$ respectively.
\begin{rem}\label{qfactor}
One useful information that the number $q$ indicates, which we shall often apply, is that
for  any slope $\g$ on $\p M$ whose distance from $\b$ is divisible by  $q$, then the representation
$\r$ factors through $\pi_1(M(\g))$, i.e. $\r(\g)=\pm I$.
\end{rem}
The following lemma can be extracted from \cite{BZ5}
specializing to exteriors of hyperbolic knots in $S^3$.

\begin{lem} \label{q}{\rm (\cite{BZ5})}
Let $M$ be the exterior of a hyperbolic knot $K$ in $S^3$ and $\b$  a finite
non-cyclic surgery slope of $K$. \\
$(1)$ If $\b$ is $D$-type, it is actually $D(2)$-type and $\b$ is an integer divisible by $4$. \\
$(2)$ If $\b$ is $T$-type, it is actually $T(3)$-type and $\b$ is an
integer or half-integer whose meridian coordinate is an odd integer divisible by $3$. \\
$(3)$ If $\b$ is $I$-type, it is  $I(2)$-, $I(3)$-, or $I(5)$-type and $\b$ is an
integer or half-integer whose meridian coordinate is relatively prime to $30$.
\\
$(4)$ If $\b$ is $O$-type, it is  $O(2)$- or $O(4)$-type and $\b$ is an
even integer not divisible by  $4$.
\end{lem}

 By a  curve in an algebraic set, we mean  an irreducible $1$-dimensional algebraic subset.
  It is known (e.g. \cite{BZ2}) that any curve $X_0$ in $X(M)$ belongs to one of the following
 three mutually exclusive types:
\newline
(a) for each slope $\g$ on $\p M$, the function  $f_\g$ is constant on $X_0$;
\newline
 (b) there is a unique slope $\g_0$ on $\p M$ such that the function $f_{\g_0}$ is constant on $X_0$;
\newline
 (c) for each slope $\g$ on $\p M$, the function $f_\g$ is non-constant on $X_0$.
\newline
 We call a curve of $X(M)$ in case (a) a constant curve, in case (b) a semi-norm curve
 and in case (c) a norm curve.  Indeed as the names indicate,  a semi-norm curve or norm curve
 in $X(M)$  can be used to define  a semi-norm or norm respectively on the real $2$-dimensional
 plane $H_1(\p M; \mathbb R)$ satisfying certain properties.
 In this paper we only need to consider those curves in $X(M)$ which are irreducible components of $X(M)$
 and  contain irreducible characters.
 We call such a curve  a   nontrivial  curve component of $X(M)$.
 In a  nontrivial curve component of $X(M)$ all but finitely many characters are
 irreducible. But to see if a curve component in $X(M)$ is nontrivial it suffices to
 check if the curve  contains at least one  irreducible character. Note that a norm curve component is
 automatically nontrivial. Also note that a norm curve component
 always exists for any hyperbolic knot exterior, but a constant curve or a nontrivial semi-norm curve may not always exist.

For a curve $X_0$ in $X(M)$, let $\tilde X_0$ be the smooth projective
completion of $X_0$ and let $\phi:\tilde X_0\ra X_0$ be the birational
equivalence.
The map $\phi$ is onto and is defined at all but finitely many points of $\tilde X_0$.
The points where $\phi$  is not defined are called ideal points.
The  map $\phi$ induces an isomorphism
from  the function field of $X_0$ to that of $\tilde X_0$. In particular
every regular function $f_\g$ on $X_0$
corresponds uniquely to its extension $\tilde f_\g$ on $\tilde X_0$ which is a rational function.
 If $\tilde f_\g$ is not a constant function on $\tilde X_0$, its degree, denoted $deg(\tilde f_\g)$,
 is equal to the number of zeros of $\tilde f_\g$ in $\tilde X_0$ counted with multiplicity, i.e.
 if $Z_x(\tilde f_\g)$ denotes the zero degree of $\tilde f_\g$ at a point $x\in \tilde X_0$, then
 $$deg(\tilde f_\g)=\sum_{x\in \tilde X_0} Z_x(\tilde f_\g).$$
Note that if $\chi_\r$ is a smooth point of $X_0$ then
$ \phi^{-1}(\chi_\r)$ is a single point and the zero degree of
$f_\g$ at $\chi_\r$ is equal to
the zero degree of $\tilde f_{\g}$ at $x=\phi^{-1}(\chi_\r)$.

From now on in this section we make the following
{\bf special assumption}: let $M$ be the exterior of  a hyperbolic knot in $S^3$
and suppose $M$ has a half-integer finite surgery slope $\a$.
It follows from \cite[Theorem 2.0.3]{CGLS} that  each of
 $\a$  and the meridian slope $\m$ is  not a boundary slope.

We shall identify $H_1(\p M, \mathbb R)$ with the real $xy$-plane
so that $H_1(\p M;\mathbb Z)$ are integer lattice points $(m,n)$
with $\mu=(1,0)$ being the meridian class and $\l=(0,1)$ the longitude class.
So each slope $p/q$ in $\p M$  corresponds to the pair of primitive elements
$\pm (p,q)\in H_1(\p M;\z)$.

 \begin{thm}\label{norm properties}
  Let $X_1$ be a  norm curve component of $X(M)$.
 Then  $X_1$  can be used to define  a norm $\|\cdot\|_\sxz$ on $H_1(\p M;\mathbb R)$, known as  Culler-Shalen  norm, with  the following properties:
\newline
(1) For each
nontrivial element $\g=(m,n)\in H_1(\p M;\z)$,
 $\|\g\|_\sxz=deg(\tilde f_\g)\ne 0$ (thus is a positive integer).
\newline
(2) The norm is symmetric to the origin, i.e. $\|(a,b)\|_\sxz=\|(-a,-b)\|_\sxz$
for all $(a,b)\in H_1(\p M;\mathbb R)$.
Let $$s_1=\textrm{min}\{\|\g\|_\sxz; \g\in H_1(\p M;\z), \g\ne 0\}$$
and $B_1$ be the set of points  in $H_1(\p M;\mathbb R)$ with norm less than or equal to
  $s_1$.
Then $B_1$  is a convex finite sided polygon symmetric to the origin
whose interior does not contain any non-zero element of $H_1(\p M;\z)$.
\newline
(3)
If $(a,b)$ is a vertex of $B_1$, then there is a boundary slope $p/q$ in $\p M$
 such that $\pm (p,q)$ lie in the line passing through $(a,b)$ and $(0,0)$.
\newline
(4) If we normalize the area of a parallelogram spanned by
any pair of generators of $H_1(\p M;\z)$ to be $1$, then $Area(B_1)\leq 4$.
\newline
(5) If $\b$ is a cyclic surgery slope but is  not a boundary slope, then
$\b\in \p B_1$ (so  $\|\b\|_\sxz=s_1$)  but is not a vertex of $B_1$.
More precisely for each non-zero element
 $\g\in H_1(\p M;\z)$ and for every point $x\in \tilde X_1$,
$Z_x(\tilde f_\b)\leq Z_x(\tilde f_\g)$. In particular the meridian slope $\m$ has this property.
\newline
(6) If $\b$ is a $T$-type finite surgery slope
but is not a boundary slope, then
$\|\b\|_{X_1}= s_1+2$ or $s_1$ corresponding to whether
the  irreducible character $\chi_\r$ of $X(M(\b))$ (given by
Lemma~\ref{characters of finite groups} (1)) is contained in $X_1$ or not
respectively.
\newline
(7) If $\b$ is an $O$-type finite surgery slope
but is not a boundary slope, then $\|\b\|_{X_1}= s_1+3$ or $s_1+2$ or $s_1+1$ or $s_1$
corresponding to whether
both of or only the $O$-type character or only the $D$-type character or neither of the two irreducible characters of $X(M(\b))$
(given by
Lemma~\ref{characters of finite groups} (2)) are or is contained in $X_1$ respectively.
\newline
(8) If $\b$ is an $I$-type finite surgery slope
but is not a boundary slope, then $\|\b\|_{X_1}=s_1+4$ or
$s_1+2$ or $s_1$  corresponding to whether
both of or only one of or neither of the two irreducible characters of $X(M(\b))$
(given by
Lemma~\ref{characters of finite groups} (3)) are or is contained in $X_1$ respectively.
\newline
(9) The half-integral finite surgery slope $\a$ is either of $T$-type or $I$-type.
\newline
(9a) If $\a$ is of $T$-type, the curve $X_1$ contains the unique
irreducible character of $X(M(\a))$,
 $\|\a\|_\sxz=\|\m\|_\sxz+2=s_1+2$, and $X(M)$ has no other  norm curve component.
\newline
(9b)  If $\a$ is of $I$-type, the curve $X_1$ contains at least one of the two
irreducible characters of $X(M(\a))$,
$\|\a\|_\sxz=\|\m\|_\sxz+2=s_1+2$ if exactly one of the two
irreducible characters of $X(M(\a))$
is  contained in $X_1$,
or $\|\a\|_\sxz=\|\m\|_\sxz+4=s_1+4$ if both of the two
irreducible characters of $X(M(\a))$ are contained in $X_1$.
Also when $\|\a\|_\sxz=\|\m\|_\sxz+4=s_1+4$, $X(M)$ does not have any other
norm curve component.
\end{thm}

Properties (1)-(5) of Theorem~\ref{norm properties}
 are originated from \cite{CGLS} and properties (6)-(9) from \cite{BZ1} \cite{BZ2}.
As (9) was not explicitly stated in \cite{BZ1} \cite{BZ2}, we give here a brief
explanation.
Since each  of $\m$  and $\a$ is  not a boundary slope,
   each of them  is not contained in a line passing a vertex of $B_1$ and the origin by (4).
  In particular each of $\m$ and $\b$ is not a vertex of $B_1$.
By (5), we have $\|\m\|_\sxz=s_1$.
We claim that $\a$ is not contained in $B_1$, i.e.
$\|\a\|_\sxz>\|\m\|_\sxz=s_1$.
For if $\a=(2p+1,2)\in B_1$, then since $\m=(1,0)\in B_1$  and since  $B_1$ is a convex set,  $B_1$ also contains
the points $(p,1)$ and $(p+1,1)$.
This  would imply that the area of $B_1$ is $\geq 4$, which by (4) would imply
that  $B_1$ is a parallelogram with $\pm \m$ and $\pm \a$ as vertices,
contradicting to our early conclusion.
Now the conclusion that $\a$ is either $T$-type or $I$-type
follows from Lemma~\ref{q} and all the conclusions of
(9a) and (9b) follow directly from \cite{BZ1} \cite{BZ2}, due essentially to the facts that each irreducible character in $X(M(\a))$ is a smooth point of $X(M)$ and
that the zero degree of $f_\a$ at such character is $2$ while the zero degree of $f_\m$
at such point is $0$.

\begin{rem}\label{zeroes of factors}
In fact properties (6)-(8) of Theorem~\ref{norm properties} are
also due to similar facts: when $\b$ is a finite non-cyclic slope of $M$
each irreducible character of $X(M(\b))$ is a smooth point of $X(M)$ \cite{BZ1}\cite{BZ2}
and thus is contained a unique component of $X(M)$, and when such character is contained in
$X_1$, then  the zero degree of $f_\b$ at such character is $2$ (except when
the character is dihedral in which case the zero degree is $1$) while the zero degree of $f_\m$
at such point is $0$.
Moreover if the character factor through $M(\g)$ for some  slope $\g$ then the zero
degree of $f_\g$ at this point is also $2$ (or $1$ when the character is dihedral).
We shall say that the character contributes to the norm of $\g$ by $2$ (or $1$)
beyond the minimum norm $s_1$. This extended property shall also be applied later in this paper.
\end{rem}

As $M$ is hyperbolic, any component $X_1$ of $X(M)$ which contains the character of a discrete
faithful representation of $\pi(M)$ is a  norm curve component of $X(M)$.
To apply Theorem~\ref{norm properties} more effectively we consider the set $C$
of all (mutually distinct)  norm curve components $X_1$,...,$X_k$  in $X(M)$
and let $\|\cdot\|=\|\cdot\|_{X_1}+\cdots\|\cdot\|_{X_k}$ be the
norm defined by $C$.
In particular $C$ contains  the orbit  of $X_1$ under the
$Aut(\c)$-action on $X(M)$ (cf. \cite[Section 5]{BZ5} for the $Aut(\c)$-action).
In fact under the special assumption that $M$ has a half-integer
finite surgery slope,  $C$ has at most two components.
Let $$s=\text{min}\{\|\g\|; \g\in H_1(M,\p M), \g\ne 0\}$$
and let $B$ be the disk in the plane $H_1(\p M;\mathbb R)$
centered at origin with radius $s$ with respect to the norm $\|\cdot\|$.
Obviously $C$, $\|\cdot\|$, $s$ and $B$ are uniquely
associated to $M$.
The following theorem  follows directly from Theorem
\ref{norm properties}.

 \begin{thm}\label{norm of finite slope}
With $C$, $\|\cdot\|$, $s$ and $B$  defined above for $M$, we have:
 \newline
(1) $s>0$ is an integer and $B$  is a convex finite sided polygon symmetric to the origin.
 \newline
(2) If $\b$ is a cyclic slope but is not a boundary slope,
then $\b\in \p B$ (so $\|\b\|=s$) but is not a vertex of $B$.
In particular $\m$ is such a slope.
\newline
(3) If $\b$ is a $T$-type finite surgery slope
but is not a boundary slope, then
$\|\b\|= s+2$ or $s$ corresponding to whether
the  irreducible character $\chi_\r$ of $X(M(\b))$ (given by
Lemma~\ref{characters of finite groups} (1)) is contained in $C$ or not respectively.
\newline
(4) If $\b$ is an $O$-type finite surgery slope, then
$\|\b\|= s+3$ or $s+2$ or $s+1$ or $s$
corresponding to whether
both of or only the $O$-type character or only the $D$-type character or neither of the two irreducible characters of $X(M(\b))$
(given by
Lemma~\ref{characters of finite groups} (2)) are or is contained in $C$ respectively.
\newline
(5) If $\b$ is an $I$-type finite surgery slope
but is not a boundary slope, then $\|\b\|=s+4$ or
$s$  corresponding to whether
both of or  neither of the two irreducible characters of $X(M(\b))$
(given by
Lemma~\ref{characters of finite groups} (3)) are contained in $C$ respectively.
\newline
(6) The half-integral finite surgery slope $\a$ is either of $T$-type or $I$-type.
\newline
(6a) If $\a$ is of $T$-type,
 then $\|\a\|=\|\m\|+2=s+2$ and the irreducible character of $X(M(\a))$ is contained in
 $C$.
\newline
(6b)  If $\a$ is of $I$-type, then
$\|\a\|=\|\m\|+4=s+4$
and both
irreducible characters of $X(M(\a))$ are contained in $C$.
\end{thm}

We only need to note that Theorem~\ref{norm of finite slope} (4) and (6b)
hold because of the $Aut(\c)$-action (cf. \cite[Remark 9.4]{BZ5}).

\begin{thm}\label{seminorm properties}
Suppose that $X_0$ is  a  nontrivial semi-norm curve component of $X(M)$
(such curve may not always exist) and let
  $\g_0$ be the unique slope such that $f_{\g_0}$ is constant on $X_0$ (we call $\g_0$
  the associated slope to $X_0$).
  The curve $X_0$ can be used to define  a semi-norm $\|\cdot\|_\sxo$ on $H_1(\p M;\mathbb R)$, called Culler-Shalen
  semi-norm, with the following properties:
\newline
 (1) For the associated slope $\g_0$, $\|\g_0\|_\sxo=0$ and $\g_0$ is a boundary slope
 of $M$.
\newline
   (2) For each slope $\g\ne \g_0$, $\|\g\|_\sxo=deg(\tilde f_\g)\ne 0$ (so is a positive integer).
\newline
  (3) For the meridian slope $\m=(1,0)$,
    $\m\ne \g_0$ and $\|\m\|_\sxo>0$ is minimal among all slopes $\g\ne \g_0$.
     More precisely for every point $x\in \tilde X_0$,
$Z_x(\tilde f_\m)\leq Z_x(\tilde f_\g)$ for each slope $\g\ne \g_0$.
Furthermore for every slope $\g$,  $\|\g\|_\sxo=\D(\g,\g_0) \|\m\|_\sxo$.
In particular $\D(\m,\g_0)=1$, i.e. $\g_0$ is an integer slope.
\newline
(4) For the half-integer finite surgery slope $\a$,
$\|\a\|_\sxo=\|\m\|_\sxo$ and thus $\D(\a,\g_0)=1$.
\end{thm}

Theorem~\ref{seminorm properties} is contained in \cite{BZ2}.
We only need to note that item (4) of Theorem~\ref{seminorm properties}
 holds because
 $X_0$ cannot contain any irreducible character of $X(M(\a))$,
  due to item (6) of Theorem ~\ref{norm of finite slope}.

\begin{lem}\label{s+2}Let $\eta$ be any one of the two integer slopes which are distance one from
the half-integer finite surgery slope
$\a$. Then $\|\eta\|\leq s+1$ if $\a$ is of $T$-type
and $\|\eta\|\leq s+2$ if $\a$ is of $I$-type.
\end{lem}

\pf Write $\a=(2p+1,2)$, then $\eta=(p,1)$ or $(p+1,1)$.
We prove the case when $\a=(2p+1,2)$ is of $I$-type and $\eta=(p+1,1)$. The other three cases
can be treated similarly.
So we have $\|\a\|=s+4$ by Theorem~\ref{norm of finite slope} (6b).
Let $B(r)$ be the norm disk in the plane $H_1(\p M;\mathbb R)$ centered at the origin with radius $r$. Then $B(s)=B$.
The point $\m=(1,0)$ lies in $\p B(s)$.
 There is a positive real number $a$ such that the point
 $(1+2a, 0)$ has   norm $\|\m\|+4=s+4$.
 By the convexity of $B(r)$  with any radius $r$,
 the line segment in the plane $H_1(\p M;\mathbb R)$ with endpoints  $(1+2a, 0)$ and $(2p+1,2)$ is
 contained in the norm disk $B(s+4)$.
 It follows that the line segment with endpoints $(1+a, 0)$ and $(p+1,1)$
 is contained in the norm disk $B(s+2)$.
 \qed

%%%%%
%%%%%
%%%%%
%%%%%
%%%%%

\section{Proof of Theorem~\ref{17/2and23/2}}\label{sect:17/2and23/2}

Suppose otherwise that $K$ is  a hyperbolic
knot in $S^3$ on which  $17/2$ or  $23/2$ is a finite surgery slope.
We will get a contradiction from this assumption.
Here is an outline of our strategy.
 Let $\a$ be the finite surgery slope  $17/2$ or $23/2$  on $K$ and let
 $\d$  be the slope
 $$\d=\left\{\begin{array}{ll}
 9,\;\;&\mbox{if $\a=17/2$},\\
 11,\;\;&\mbox{if $\a=23/2$}.
 \end{array}\right.$$ Our first  task is to show

\begin{prop}\label{non-hyper-surgery}  Dehn surgery on the given  hyperbolic knot $K$ with the slope $\d$   does not yield a hyperbolic $3$-manifold.
\end{prop}

Note that by Known Facts~\ref{known2} (4) and Table~\ref{table:Half}, the knot $K$ has  genus $2$.
It is known that Dehn surgery with the slope $\d$  on any hyperbolic
 knot in $S^3$ of genus $2$  can never produce a lens space  (Known Facts~\ref{known2} (7))
 or   a reducible  manifold \cite{MS}
 or  a toroidal manifold   \cite{Lee}.
 It also follows from Lemma~\ref{q}, Known Facts~\ref{known2} (3), Table~\ref{table:T} and Table~\ref{table:I}  that
 the $\d$-surgery on $K$ cannot yield  a spherical space form.
Therefore by Proposition~\ref{non-hyper-surgery},  the $\d$-surgery on $K$
must produce an irreducible Seifert fibred space
which has infinite fundamental group
but does not contain incompressible tori.
But this will contradict our next assertion:

 \begin{prop}\label{non-small-Seifert} For the given knot $K$,
 Dehn surgery with the slope $\d$  cannot yield an irreducible  Seifert
fibred space with infinite fundamental group but containing no incompressible tori.
\end{prop}

The rest of this section is devoted to the proofs of the above two propositions.
The main tool is the character variety method.

Recall that the Lie algebra $sl_2(\c)$ of $SL_2(\c)$  consists of all $2\times 2$  complex matrices  with zero trace.
The group $SL_2(\c)$ acts on $sl_2(\c)$ through the adjoint homomorphism
$Ad: SL_2(\c)\ra Aut(sl_2(\c))$  given by matrix conjugation.
As $-I$ acts trivially on $sl_2(\c)$,  the adjoint action of $SL_2(\c)$ on $sl_2(\c)$
factors  through $PSL_2(\c)$.
If $\r\in R(\G)$ is a $PSL_2(\c)$ representation of a group $\G$,
we use  $Ad\circ\r$ to denote the induced action of $\G$ on $sl_2(\c)$.
Let $H^1(\G; Ad\circ\r)$ denote the   group cohomology with respect to the module
 $Ad\circ\r:\G\ra Aut(sl_2(\c))$.
Note that for a connected compact manifold $Y$ and $\r\in R(Y)$, $H^1(Y; Ad\circ\r)\cong
H^1(\pi_1(Y); Ad\circ\r)$.

\begin{lem} \label{smooth degree 2}Let $M$ be the exterior of a knot in $S^3$.
Suppose that for some slope $\b$, $X(M(\b))$ has a character $\chi_{\r_0}$ satisfying:
\newline(i) $\chi_{\r_0}$ is strictly irreducible,
\newline
(ii)  $H^1(M(\b); Ad\circ\r_0)=0$,
\newline
(iii) the image of $\r_0$ does not contain parabolic elements.
\newline
Then the following conclusions hold:
\newline
(1) As a point in $X(M(\b))$, $\chi_{\r_0}$ is  an $0$-dimensional  algebraic component
 of  $X(M(\b))$, and as a point in $X(M)$, $\chi_{\r_0}$  is a smooth point
  and is contained in a
  unique  curve component $X_0$ of $X(M)$.
\newline
(2) For the curve component $X_0$ given in (1), the function $f_\b$ is not constant on $X_0$.
So in particular $X_0$ is not a constant curve.
\newline
(3) The point $\chi_{\r_0}$ is a zero point of $f_\b$ but is not a zero point of $f_\m$,
and moreover the zero degree of $f_\b$ at $\chi_{\r_0}$ is $2$.
\end{lem}

\pf The conclusion of part (1)  follows from conditions (i) and (ii) and
is a special case of  \cite[Theorem 3]{BZ3} (although
the theorem  there was stated for  $SL_2(\c)$-representations, the same proof
applies to $PSL_2(\c)$-representations).

The idea of proof for part (2)   is essentially contained in \cite{BZ4} for a similar situation in $SL_2(\c)$-setting.
For the reader's convenience we give a proof for our current situation.
 Suppose otherwise that  $f_\b$ is constant on $X_0$. Then it is
 constantly zero on $X_0$ since $\chi_{\r_0}$ is obviously a zero
point of $f_\b$.
So for every $\chi_\r\in X_0$, $\r(\b)$ is either $\pm I$   or a parabolic element.
Note that $\r(\b)$ cannot be $\pm I$ for all $\chi_\r\in X_0$, for otherwise
$X_0$ becomes a curve in $X(M(\b))$ containing $\chi_{\r_0}$, which contradicts the fact that
$\chi_{\r_0}$ is an isolated point in $X(M(\b))$.
Therefore $\r(\b)$ is parabolic   for all
but finitely many points $\chi_\r$ in  $X_0$.
As the meridian $\m$ commutes with $\b$ in $\pi_1(M)$, $\r(\m)$ is either $\pm I$ or
parabolic for all but finitely many points $\chi_\r\in X_0$. Hence
$f_\m$ is also constantly zero. In particular, $\chi_{\r_0}$ is a zero point of
$f_\m$. But $\r_0(\m)$ cannot be $\pm I$, so $\r_0(\m)$ is parabolic.
This violates  condition (iii).

 As we have seen in the proof of part (2),
$\chi_{\r_0}$ is a zero point of $f_{\b}$ but cannot be a zero point of $f_\m$.
Combined with  condition (i),
the conclusion of part (3) now follows from \cite[Theorem 2.1 (2)]{AB}.\qed

Let $M$ be the exterior of the given hyperbolic knot $K$ and
$\|\cdot\|$ be the total Culler-Shalen norm
on $H_1(\p M; \mathbb R)$  defined in Section~\ref{sect:CS norm}.
Recall $$s=\text{min}\{\|\g\|; \g\in H_1(\p M;\z), \g\ne 0\}$$
Let $B(r)$ be the norm disk in the plane $H_1(\p M; \mathbb R)$
centered at the origin of radius $r$.
We already knew
that each of $\m$ and $\a$ is not a boundary slope.
By Lemma~\ref{q}, $\a$ is an $I$-type finite surgery slope,
and by Theorem~\ref{norm of finite slope}
$$
\|\m\|=s,\;\;\; \|\a\|=s+4.
$$
By Lemma~\ref{s+2}, we have
\begin{equation}\label{d<s+2}\|\d\|\leq s+2.\end{equation}
Proposition~\ref{non-hyper-surgery} follows from (\ref{d<s+2}) and the following proposition.

\begin{prop}\label{hyperbolic larger than s+2} If $M(\d)$ is a hyperbolic $3$-manifold,
then  $$\|\d\|\geq s+4.$$
\end{prop}

\pf   Note that by Mostow rigidity the closed hyperbolic  $3$-manifold $M(\d)$ has exactly two
 discrete  faithful $PSL_2(\c)$-representations $\r_1$ and $\r_2$, up to conjugation.
   Obviously the two distinct characters $\chi_{\r_1}$ and $\chi_{\r_2}$
   satisfy conditions (i) and (iii) of Lemma~\ref{smooth degree 2}.
   Both of them also satisfy condition (ii) of Lemma~\ref{smooth degree 2}, which is proved in \cite{W}.
(cf. \cite[Corollary 5]{BZ3}).
Hence Lemma~\ref{smooth degree 2} applies:
each  $\chi_{\r_i}$   lies in a unique
curve component $X_i$ of $X(M)$ ($X_1=X_2$ is possible) as a smooth point,
$X_i$ is not a constant curve and the zero degree of $f_\d$ at $\chi_{\r_i}$ is $2$
but $f_\m$ is not zero valued at $\chi_{\r_i}$.

{\bf Claim}. Each $X_i$ is a  norm curve component of $X(M)$.

We  just need to show that each $X_i$ is not a semi-norm curve.
Suppose otherwise that some $X_i$ is a
semi-norm curve. Let $\g_i$ be the associated slope  and  $\|\cdot\|_\sxi$  the corresponding
Culler-Shalen semi-norm.

By Lemma~\ref{smooth degree 2} (2)  $\g_i\ne \d$  and by Theorem~\ref{seminorm properties} (3)
$\m\ne \g_i$, $\m$ has the minimal Culler-Shalen  semi-norm among all slopes $\g\ne \g_i$ and $\g_i$ is an integer slope.
Moreover by Theorem~\ref{seminorm properties} (4),
$\g_i$ must be the slope $8$ if $\a=17/2$ or the slope
$10$ if $\a=23/2$. So we have $\D(\g_i,\d)=1$, which implies
that  $\|\d\|_\sxi=\|\m\|_\sxi$ by Theorem~\ref{seminorm properties} (3).
But at $\chi_{\r_i}$, $f_\d$ is zero  and $f_\m$ is non-zero.
Hence it follows from  Theorem~\ref{seminorm properties} (3) that
 $\|\d\|_\sxi$ is strictly larger than $\|\m\|_\sxi$.
 We get a contradiction and the claim is proved.

Hence each $X_i$ is a  norm curve component
of $X(M)$ and
thus is a member  of the
set $C$ which is the union of all norm curve components
of $X(M)$.
In particular both $\chi_{\r_1}$ and $\chi_{\r_2}$ are contained
in $C$ which by Lemma~\ref{smooth degree 2} (3) and Theorem~\ref{norm properties} (5)
implies that $\|\d\|\geq \|\m\|+4=s+4$.
 This completes the proof of the proposition.
\qed

Proposition~\ref{non-small-Seifert} follows from (\ref{d<s+2}) and the following proposition.

\begin{prop}\label{ssfs larger than s+2} If $M(\d)$  is an irreducible Seifert fibred space
with infinite  fundamental group but  containing no incompressible tori,
then  $$\|\d\|\geq s+4.$$\end{prop}

\pf
Since $M(\d)$ is an irreducible Seifert fibred  space with infinite fundamental group
 but containing no incompressible tori, its base orbifold is a $2$-sphere with three cone points
whose cone orders do not form an elliptic triple. So the base orbifold is $S^2(a,b,c)$
and $\displaystyle \frac{1}{a}+\frac{1}{b}+\frac{1}{c}\leq 1$.
Note that the fundamental group of $\pi_1(M(\d))$ surjects onto
the orbifold fundamental group of $S^2(a,b,c)$ which is the triangle group
$$\bigtriangleup(a,b,c)=<x,y; x^a=y^b=(xy)^c=1>.$$

We may assume that $a\geq b\geq c\geq 2$.
Note that $M(\d)$ has cyclic first homology of odd order. It follows that
$gcd(a,b,c)=1$,  $b\geq 3$, $a\geq 5$, and at most one of $a,b,c$ is even.
In particular  $S^2(a,b,c)$ must be a hyperbolic $2$-orbifold
and thus $\bigtriangleup(a,b,c)$ has a discrete faithful
representation $\r_1$ into $PSL_2(\mathbb R)\subset PSL_2(\c)$. Therefore
$\r_1$ is strictly irreducible and its image group does not contain parabolic elements.

On the other hand,  applying \cite[Addendum on page 224]{Boyer}, we see
that the  triangle group
$\bigtriangleup(a,b,c)$ has a non-abelian  representation $\r_2$ into
$SO(3)\subset PSL_2(\c)$.
Thus $\r_2$ is irreducible and its image does not contain parabolic elements.
It is easy to check that $\r_2$ is also strictly irreducible for otherwise
$\r_2$ would be a dihedral representation, contradicting the fact that
$M(\d)$ has odd order  first homology.

Evidently  $\r_1$ is not conjugate to $\r_2$.
So we have two distinct characters $\chi_{\r_1}$ and $\chi_{\r_2}$
in $X(\bigtriangleup(a,b,c))\subset X(M(\d))\subset X(M)$.
As points in $X(M(\d))$, the two  characters $\chi_{\r_1}$ and $\chi_{\r_2}$
both satisfy conditions (i) and (iii) of Lemma~\ref{smooth degree 2}.
They also both meet condition (ii) by \cite[Proposition 7]{BZ3} (although the
result   there is stated for $SL_2(\c)$-representations, similar argument
works  for $PSL_2(\c)$-representations).
Hence Lemma~\ref{smooth degree 2} applies:
each  $\chi_{\r_i}$   lies in a unique
curve component $X_i$ of $X(M)$ ($X_1=X_2$ is possible) as a smooth point,
$X_i$ is not a constant curve and the zero degree of $f_\d$ at $\chi_{\r_i}$ is $2$
but $f_\m$ is not zero valued at $\chi_{\r_i}$.

We can now argue  exactly as in the proof of Proposition~\ref{hyperbolic larger than s+2}
to show that each $X_i$ is a  norm curve component of $X(M)$ which leads to the conclusion
 that $\|\d\|\geq s+4$.
\qed

%%%%%
%%%%%
%%%%%
%%%%%
%%%%%

\section{Proof of Theorem~\ref{main thm}--Part (I)}\label{sect:main thm-part I}

In this section we prove Theorem~\ref{main thm}  in case there is a half-integer
finite surgery slope on the given knot $K$. We actually show the following
theorem which provides more information in this case.

\begin{thm}\label{bound nonintegral}
Suppose that  $K$ is  a hyperbolic knot in $S^3$
which admits a half-integer  finite surgery slope $\a$.
 \newline
 (1) $\a$ is one of the following slopes:
 $$\begin{array}{l}
\{43/2, [11,2;3,2]\},
\{45/2, [11,2;3,2]\},
\{51/2, [13,3;3,2]\},
\{53/2, [13,3;3,2]\},\\
\{77/2, [19,2;5,2]\},
\{83/2, [21,2;5,2]\},
\{103/2, [17,3;3,2]\},
\{113/2, [19,3;3,2]\}.\end{array}$$
Here  each sample knot attached to a slope in the list   plays the role as before:  the same surgery slope
on the sample knot yields the same spherical space form, and $K$  has  the same knot Floer homology as the sample  knot.
 \newline
 (2) There is at most one other nontrivial  finite surgery slope $\b$, and if there is one,
it is an integer slope distance one from $\a$.
The only possible pairs for such $\a$ and $\b$ are:$$
\{43/2,21, [11,2;3,2]\},
\{53/2,27, [13,2;3,2]\},
\{103/2,52, [17,3;3,2]\},
\{113/2,56, [19,3;3,2]\}$$
when $\b$ is non-cyclic
and
$$
\{45/2,23, [11,2;3,2]\},
\{51/2,25, [13,2;3,2]\},
\{77/2,39, [19,2;5,2]\},
\{83/2,41, [21,2;5,2]\}
$$
when $\b$ is cyclic.
Here  each sample knot attached to a pair  plays the role as before:  the same surgery slopes
 on the sample knot yield the same spherical space forms, and $K$  has  the same knot Floer homology as the sample  knot.
\end{thm}

The proof uses mainly character variety techniques, based on Known Facts~\ref{known1} and~\ref{known2}.
 First we need to prepare a few more  lemmas.

\begin{lem}\label{zeros from reducibles}
Let $K$ be a knot in $S^3$. Suppose  the Alexander polynomial $\D_K(t)$ of $K$ has a simple root $\xi=e^{i\theta}$ on the unit circle in the complex plane of order $n$.
Then the knot exterior $M$ of $K$ has a reducible non-abelian $PSL_2(\c)$ representation $\r$ such that
\newline
(1) $\r(\l)=\pm I$ and $\r(\m)$ has order $n$.
\newline
(2) The character $\chi_\r$ of $\r$ is contained in a
unique nontrivial curve component $X_0$ of $X(M)$ and is a smooth point
of $X_0$. Moreover $X_0$ is either a semi-norm curve or a norm curve.
In fact $f_\m$ is non-constant on $X_0$.
\newline
(3) For any slope $\g$ on $\p M$, if $f_\g$ is non-constant on $X_0$
and if the reducible non-abelian character $\chi_\r$ is a zero point of $f_\g$, then the zero degree of $f_\g$ at $\chi_\r$ is at least $2$.
\end{lem}

\pf (1) It was known long time ago \cite{dR} \cite{Bu} that the exterior of a knot $K$ in $S^3$ has  a reducible, non abelian $PSL_2(\c)$ representation
$\r$ with $\r(\l)=\pm I$ and $\r(\m)=\pm\left(\begin{array}{cc}
a&0\\0&a^{-1}\end{array}\right)$
if and only if $\D_K(a^2)= 0$.
Hence the conclusions of (1) follow from the given conditions.

(2) It was shown in \cite{FK} that the given reducible character $\chi_\r$ is an endpoint  of a  (real) curve of irreducible  $SO(3)$ characters and also  an endpoint of a (real) curve of irreducible $PSL_2(\mathbb R)$ characters.
 Furthermore from the arguments
 of \cite{FK} one sees that on these  curves, the function $f_\m$ is non-constant.
 Later on in \cite{HPP} it was shown that for the given reducible non-abelian representation $\r$,  the space of group $1$-cocycles $Z^1(\pi_1(M), Ad\circ \r)$
 is $4$-dimensional,
  $\r$ is contained in a unique $4$-dimensional component $R_0$ of $R(M)$
 as a smooth point,  $\chi_\r$ is contained in a unique $1$-dimensional  nontrivial component  $X_0$ of $X(M)$ and is a smooth  point of $X_0$ (although in \cite{HPP} the above conclusions are  given for $SL_2$ representation and character varieties, the same conclusions also hold in $PSL_2$ setting. See \cite[Theorem 4.1]{Boyer2} \cite{HP}).
 So the conclusions of (2) hold.

 (3) First note that since  $\dim_\c Z^1(\pi_1(M), Ad\circ \r)=\dim_\c R_0=4$,
   the Zariski tangent space of $R_0$ at $\r$ can be identified with
   $Z^1(\pi_1(M), Ad\circ \r)$.
  By \cite[Theorem 4.1]{Boyer2}, there is an analytic $2$-disk $D$
  smoothly embedded in $R_0$ containing $\r$ such that
  $D\cap t^{-1}(\chi_\r)=\{\r\}$ and $t|D$ is an analytic  isomorphism
  onto a smooth  $2$-disk neighborhood of $\chi_\r$ in $X_0$.
  One can choose a smooth path   $\r_s$ in $D\subset R_0$ depending differentiably on a real parameter $s$ close to $0$, passing through $\r$ at $s=0$,
 of the form
 $$\r_s=\pm \exp(su+O(s^2))\r$$
  for some $u\in Z^1(\pi_1(M), Ad\circ \r)$ (see \cite{G}).
Now letting $\s(s)=t(\r_s)\subset X_0$ and calculating as in \cite[Section 4]{BZ1}, we get
$$f_\g(\s(s))=[trace(\r_s(\g))]^2-4=2trace(u(\g)^2)s^2+O(s^3)$$
which implies that
the zero degree of $f_\g$ at $\chi_\r$ is at least $2$ (applying \cite[Lemma 4.8]{BZ1}). Here we have used the fact that $\r(\pi_1(\p M))$ is a cyclic
group of finite order and thus $f_\g(\chi_\r)=0$ means $\r(\g)=\pm I$.
\qed

\begin{rem}\label{distinct characters}
Distinct  roots of $\D_K(t)$ lying on the upper half unit circle of the complex plane
give rise distinct reducible non-abelian characters because these characters have
distinct real values in $(0,1)$ when valued on the meridian $\m$ of the knot.
\end{rem}

\begin{rem}\label{non-abelian characters and norm}
The curve $X_0$ given in Lemma~\ref{zeros from reducibles} (2)
is  either a norm or semi-norm curve on which $f_\m$ is non-constant.
Note that  the reducible non-abelian character $\chi_\r$ given in Lemma
\ref{zeros from reducibles} is not a zero point of $f_\m$.
Now suppose that $\m$ is not a boundary slope, then it has the minimal
norm or semi-norm.
So if for some slope $\g$, $f_\g$ is non-constant on $X_0$
and $f_\g(\chi_\r)=0$, then the point $\chi_\r$ contributes
to  the norm or semi-norm of $\g$ at least by $2$ beyond the norm or semi-norm of $\m$.
Also note that $\chi_\r$ is a zero of $f_\g$ iff   the meridian coordinate
of $\g$ is divisible by $n$ (which is the order of $\r(\m)$).
\end{rem}

\begin{lem} \label{I-type}
Let $M$ be the exterior of a hyperbolic knot $K$ in $S^3$. Suppose that $M$ admits two
$I$-type surgery slopes $\b_1$ and $\b_2$.
Then as points in $X(M)$, the set of two irreducible
 characters of $X(M(\b_1))$ is equal to the set of two irreducible characters of $X(M(\b_2))$
 (cf. Lemma~\ref{characters of finite groups} (3)).
 Hence in particular $\b_1$ and $\b_2$ are of the same $I(q)$-type,  $q$ divides $\D(\b_1,\b_2)$ and $q=2$ or $3$.
Also $q=3$ if and only if one of $\b_1$ and $\b_2$ is half-integral.
\end{lem}

\pf Because  the fundamental group of any $I$-type spherical space form
is of form $I_{120}\times \z_j$ and because any
irreducible  $PSL_2(\c)$ representation of $I_{120}\times \z_j$
kills the factor $\z_j$ and
sends the factor $I_{120}$ (the binary icosahedral group) onto $I_{60}$ (the icosahedral group),
 the first  conclusion of the lemma follows.
 The rest of conclusions of the lemma follow from Lemma~\ref{q} (3) and the distance bound $3$
 for finite surgery slopes on hyperbolic knots (Known Facts~\ref{known1} (2)).
 \qed

\begin{lem}\label{DOC}
(1) If a knot $K$ in $S^3$ admits a $D$-type finite surgery, then $\det(K)>1$.
\newline
(2)  If a knot $K$ in $S^3$ admits an $O$-type finite surgery, then $\det(K)=3$.
\newline
(3) If a knot $K$ in $S^3$ admits a cyclic surgery slope with even
meridian coordinate, then $\det(K)=1$.
 \end{lem}

 \pf The lemma  follows from \cite[Theorem 10]{K} which states that for any knot $K$ in $S^3$, its knot group has precisely
$(\det(K)-1)/2$ distinct $PSL_2(\c)$ dihedral representations, modulo conjugation,
and moreover any such representation will kill any slope
with even meridian coordinate.
\qed

Now we proceed to prove Theorem~\ref{bound nonintegral}.
Part (1) of the theorem is just the combination of
Known Facts~\ref{known2} (4) and Theorem~\ref{17/2and23/2}.
So the slope $\a$ is one of the $8$ elements in $$\{43/2,45/2,51/2,53/2,77/2,83/2,103/2,113/2\}.$$
We divide the proof of part (2) correspondingly into $8$ cases.
If $\b$ is another nontrivial finite surgery slope on $K$,
then $\b$  must be an integer slope by Known Facts~\ref{known1} (2).
We shall show in each of the $8$ cases,
\newline
(i) the assumption $\D(\a,\b)>1$ will lead to a contradiction.
\newline
(ii) there is at most one such $\b$
and  $\{\a,\b\}$ is one of the pairs
listed  in Theorem~\ref{bound nonintegral}.
Each attached  sample knot
has the said properties  will also be checked.

Before we get into the cases, we make some general notes that apply to
every case.
If $\D(\a,\b)>1$, then $\D(\a,\b)\geq 3$, and
thus  $\D(\a,\b)=3$ and $\b$ is non-cyclic by Known Facts~\ref{known1} (2) (3).
Also if $\D(\a,\b)>1$, then $\b$ is not a boundary slope
by \cite[Theorem 2.0.3]{CGLS}. Recall that by the same reason, each of
$\a$ and $\m$ is not a boundary slope.
By Lemma~\ref{q}, $\a$ is either a $T$-type or $I$-type slope.
In fact $\a$ is of $T$-type iff its meridian coordinate is divisible by $3$.
Let  $\|\cdot\|$  be the total Culler-Shalen norm defined by the norm curve  set $C$ and $B(r)$ the norm disk of radius $r$.
By Theorem~\ref{norm of finite slope}, $\|\m\|=s$ has minimal norm among all slopes,
 $\m\in \p B(s)$  but is not a vertex of $B(s)$,
 $\|\a\|=s+2$ if $\a$ is of $T$-type and
 $\|\a\|=s+4$ if $\a$ is of $I$-type.

{\bf  Case 1. $\a=43/2$}.

By Known Facts~\ref{known2} (4) and Table~\ref{table:Half}  $K$ has the same Alexander polynomial as $[11,2;3,2]$
which is $$\Delta_K(t)=\Delta_{T(11,2)}(t)\Delta_{T(3,2)}(t^2).$$
In particular $\det (K)=\det(T(11,2))=11$.

If $\D(\a,\b)=3$, then $\b$ is either $20$ or $23$.
If $\b=23$, then by Lemma~\ref{q}, $\b$ is of $I$-type.
But this is impossible by Known Facts~\ref{known2} (3), Table~\ref{table:I} and
 Lemma~\ref{different A-poly}.
If $\b=20$, then $\b$ is of $D$-type by Lemma~\ref{q}.
But this is impossible by  Known Facts~\ref{known2} (5), Table~\ref{table:D} and
Lemma~\ref{different A-poly}.

So $\D(\a,\b)=1$ and $\b$ is either $21$ or $22$.
By Lemma~\ref{q} and Lemma~\ref{DOC}, $22$ cannot be
 a finite surgery slope for $K$.
So the only possible value for $\b$ is $21$.

{\bf Claim}. $\b=21$ cannot be a cyclic slope on $K$.

By Known Facts~\ref{known2} (6), if $21$ is a cyclic slope for $K$, then there is
  a Berge knot $K_0$ on which $21$ is also a cyclic slope and $K_0$ has the same Alexander polynomial
  as $K$ and thus as $[11,2;3,2]$.
  By \cite[Table~1]{Berge}, $K_0$ is not hyperbolic and thus is a torus knot or a cable over torus knot.
 From the Alexander polynomial we see that  $K_0$ has to be $[11,2;3,2]$.
 But then $21$ is not a cyclic slope for $[11,2;3,2]$ (by, e.g.  \cite[Table~1]{BH}).
 The claim is proved.

 So $\b=21$ can only possibly be  a $T$-type slope by Lemma~\ref{q}.
 Finally we note that $21$ is a $T$-type slope for $[11,2;3,2]$ (see Table~\ref{table:T}).
 Thus in this case we arrive at
 $\{43/2,21,[11,2;3,2]\}$.  Theorem~\ref{bound nonintegral} (2) is proved in this case.

{\bf  Case 2. $\a=45/2$}.

Then $\a$ is a $T$-type finite surgery slope and
$K$ has the same Alexander polynomial as $[11,2;3,2]$,
which is $$\Delta_K(t)=\Delta_{T(11,2)}(t)\Delta_{T(3,2)}(t^2).$$

If $\D(\a,\b)=3$, $\b$ is either $21$ or $24$.
If $\b=24$, then by Lemma~\ref{q}, $\b$ is of $D$-type.
But this is impossible by Table~\ref{table:D} and Lemma~\ref{different A-poly}.

When $\b=21$, it is a $T$-type slope and $\|\a\|=s+2$.
As $\D(\a,\b)=3$, we have $\|\b\|=\|\a\|$ by Lemma~\ref{q} and Theorem~\ref{norm of finite slope} (3).
By the convexity of the norm disk of any radius,
the line segment with endpoints
$(21,1)$ and $(45,2)$ is contained in the norm disk of radius $s+2$ and in particular the midpoint  $(33,\frac32)$ of the segment has norm less than or equal to $s+2$.
As  $\|(33,\frac32)\|=\frac32\|(22,1)\|\geq \frac32s$, we have $\frac32s\leq s+2$ from which we get
\begin{equation}\label{case 2, ine1}s\leq 4.
\end{equation}

The integer slope $23$ is distance $1$ from $\a$
and thus
 $\|(23,1)\|\leq s+1$ by Lemma ~\ref{s+2}. As $s\leq 4$, the norm of the point
 $(1.25,0)$ is at most $s+1$.
 Hence the line segment with endpoints $(23,1)$ and $(1.25,0)$ is contained in $B(s+1)$.
 It follows that the line segment with endpoints $(24,1)$ and $(2.25,0)$
 is contained in $B(2s+1)$.
In particular we have
\begin{equation}\label{case 2, ine2}\|(24,1)\|\leq 2s+1.
\end{equation}

On the other hand the roots of
the Alexander polynomial $\D_K(t)$  of $K$ are all simple  and are
all roots of unity. The factor $\D_{T(3,2)}(t^2)$
of  $\D_K(t)$ contains three roots  on the upper-half unit circle
in the complex plane: $e^{\pi i/3}$, $e^{\pi i/6}$, $e^{5\pi i/6}$, of
order $6$, $12$ and $12$ respectively, so
 the corresponding three distinct reducible non-abelian
 $PSL_2(\c)$ characters of $X(M)$
 factor through $M(24)$ (see Lemma~\ref{zeros from reducibles} and Remarks~\ref{distinct characters} and~\ref{non-abelian characters and norm}).
If they  are all contained in the norm  curve set $C$
then they will contribute to the norm of $(24,1)$ by at least  $6$
 beyond $s=\|\m\|$ (Lemma~\ref{zeros from reducibles} (3)).
Hence we have
\begin{equation}\label{case 2, ine3}\|(24,1)\|\geq s+6.
\end{equation}

Combining (\ref{case 2, ine2}) and (\ref{case 2, ine3}) we have
$s+6\leq 2s+1$, i.e. $s\geq 5$, which contradicts (\ref{case 2, ine1}).

Hence at least one of the above three  reducible non-abelian characters, which we denote by
$\chi_{\r_0}$, is not contained in $C$. By Lemma
\ref{zeros from reducibles} (2) and the definition of $C$ we see that  $\chi_{\r_0}$ is contained in a nontrivial semi-norm curve component $X_0$.
Now by Theorem~\ref{seminorm properties} (3) (4) we have  that the associated boundary slope
$\g_0$ of $X_0$ is an integer and $\D(\a,\g_0)=1$.
Hence $\g_0=22$ or $23$.
But $f_{\g_0}$ must be constantly zero on $X_0$
(otherwise $\tilde f_\m$ would have larger zero degree
than $\tilde f_{\g_0}$ at some point of $\tilde X_0$ which is impossible by \cite[Proposition 1.1.3]{CGLS})
and in particular is zero valued  at the non-abelian reducible character $\chi_{\r_0}$ which implies that $\g_0$  is divisible by $6$ (Remark~\ref{non-abelian characters and norm}).
We arrive at a contradiction.

So $\D(\a,\b)=1$ and $\b$ is either $22$ or $23$.
For the same reasons as given in Case 1, $22$  cannot be a finite surgery slope and $23$
cannot be a non-cyclic finite surgery slope.
So $\b=23$ is possibly a cyclic slope for $K$.
In fact, $23$ is a cyclic slope for $[11,2;3,2]$.
Hence in Case 2,  we arrive at the pair $\{45/2, 23\}$
with the sample knot $[11,2;3,2]$.

{\bf  Case 3. $\a=51/2$}.

This case can be handled very similarly  as in Case 2, and
we get the pair $\{51/2, 25\}$, where $25$ is a possible cyclic slope, with
$[13,2;3,2]$ as a sample knot.

{\bf  Case 4. $\a=53/2$}.

In this case $K$ has the same Alexander polynomial as the sample knot $[13,2;3,2]$.
When $\D(\a,\b)=3$, $\b$ is either $25$ or $28$.
By Lemma~\ref{q}, $25$
cannot be a  finite surgery slope.
 Using Lemma~\ref{q}, Table~\ref{table:D} and Lemma~\ref{different A-poly}, we can easily rule out  $\b=28$.
So $\b=26$ or $27$.
 By Lemma~\ref{q} and Lemma~\ref{DOC} (2), $\b=26$ cannot be a non-cyclic finite surgery slope
 and by Lemma~\ref{DOC} (3), $\b=26$ cannot be a cyclic surgery slope.
 Hence $\b=27$ which is a $T$-type slope for the  sample knot
 $[13,2;3,2]$. So we just need to rule out the possibility that $\b=27$ be a cyclic slope for $K$.
 This can be done by checking  that there is no Berge knot which has $27$ as cyclic slope and has the same Alexander polynomial as $[13,2;3,2]$.
 So in this case we get the member $\{53/2,27, [13,2;3,2]\}$.

{\bf  Case 5. $\a=77/2$}.

The argument is pretty much similar to that
for Case 2.
The knot $K$ has the same Alexander polynomial as $[19,2;5,2]$,
which is $$\Delta_K(t)=\Delta_{T(19,2)}(t)\Delta_{T(5,2)}(t^2).$$
In particular $\det(K)=\det(T(19,2))=19$.

When $\D(\a,\b)=3$,  $\b$ is either $37$ or $40$.
 The Alexander polynomial  of $K$ has $6$ simple roots
 provided by the  factor $\D_{T(5,2)}(t^2)$
 on the upper-half unit circle:
 $e^{k\pi i/10}, k=1,2,3,6,7,9$
 of order $10$ or $20$
 and they give rise  $6$ reducible non-abelian
 $PSL_2(\c)$ characters all of which factor through $M(40)$.
Hence $40$ cannot be a finite surgery slope for $K$.

So suppose  $\b=37$.
Both $\a$ and $\b$ are of   $I$-type.
As  $\|\a\|=s+4$,  $\|\b\|=s+4$ by Lemma~\ref{I-type}.
So the  line segment with endpoints
$(37,1)$ and $(77,2)$ is contained in $B(s+4)$ and in particular the midpoint  $(57,\frac32)$ of the segment
has norm less than or equal to $s+4$.
As $\|(57,\frac32)\|=\frac32\|(38,1)\|\geq \frac32s$,
 we have $\frac32s\leq s+4$ from which we get
$s\leq 8$.

The integer slope $39$ is distance $1$ from $\a$ and thus
 $\|(39,1)\|\leq s+2$ by Lemma~\ref{s+2}.
Hence the line segment with endpoints
$(39,1)$ and $(1.25,0)$ is contained in the norm disk of radius $s+2$.
 It follows that the line segment with endpoints $(40,1)$ and $(2.25,0)$
 is contained in $B(2s+2)$.
In particular $\|(40,1)\|\leq 2s+2$.

On the other hand
if the above $6$ reducible non-abelian characters are all  contained in the norm curve set $C$
then they would contribute to the norm of $(40,1)$ by at least  $12$
 beyond $s$, i.e.  $\|(40,1)\|\geq s+12$.
Hence combining the last two inequalities we have
$s+12\leq 2s+2$, i.e. $s\geq 10$.
We arrive at a contradiction with the early inequality  $s\leq 8$.

So at least one of the above $6$ reducible non-abelian characters is not  contained in  $C$
in which case we can get a contradiction
exactly as in Case 2.

Hence  if $\b$ is another nontrivial finite surgery slope,
then $\D(\a,\b)=1$ and $\b$ is either $38$ or $39$.
By Lemma~\ref{q} and Lemma~\ref{DOC} (2) (3), $38$ cannot be a finite surgery slope for $K$.
  If $39$ is a finite surgery slope, it cannot be non-cyclic by Lemma~\ref{q} and Table~\ref{table:T}.
It could be a cyclic slope for $K$.
In fact it is a cyclic slope of $[19,2;5,2]$ by \cite[Table~1]{BH}.
So in this case, $\a=77/2$  and $\b=39$ are the only possible finite surgery slopes
for $K$ (the former an $I$-type and the latter a $C$-type), with $[19,2;5,2]$ as a sample knot.

{\bf  Case 6. $\a=83/2$}.

This case can be treated very similarly  as in Case 5, and
 $\a=83/2$ and $\b=41$  are the only possible finite surgery  slopes for $K$
 ($\a$ an $I$-type and $\b$  a $C$-type),  with $[21,2;5,2]$ as
a sample knot.

{\bf  Case 7. $\a=103/2$}.

This is perhaps  the hardiest case.
We know that  $\a$ is of $I$-type and $K$ has the same Alexander polynomial as $[17,3;3,2]$,
which is $$\Delta_K(t)=\Delta_{T(17,3)}(t)\Delta_{T(3,2)}(t^3).$$
When $\D(\a,\b)=3$,  $\b$ is either $50$ or $53$.
By Lemma~\ref{q} and Table~\ref{table:I}, $53$ cannot be a finite surgery slope for $K$.
If $\b=50$ is a finite surgery slope, it is $O$-type by Lemma~\ref{q}.
We  have $\|\a\|= s+4$
and $\|\b\|\leq s+3$ by Theorem~\ref{norm of finite slope} (4) (6).
So the  line segment with endpoints
$(50,1)$ and $(103,2)$ is contained in the norm disk of radius $s+4$ and moreover the midpoint  $\frac 32(51,1)$ of the segment is contained in  the interior of $B(s+4)$ thus
has norm less than $s+4$.
But $\frac32\|(51,1)\|\geq \frac32s$.
So we have $\frac32s<s+4$ from which we get
$s\leq 7$.

As $K$ is fibred (Known Facts~\ref{known2} (1)), we may apply \cite[Theorem 5.3]{GO} which asserts that  there is an essential lamination in $M$  with a degenerate  slope $\g_0$ such that $M(\g)$ has an essential lamination and thus has infinite fundamental group if $\D(\g,\g_0)>1$.
  Hence  $\g_0$ must be the slope
 $51$.
  Furthermore by \cite[Theorem 2.5]{Wu} combined with the geometrization theorem of Perelman,
  $M(\g)$ is hyperbolic if $\D(\g, \g_0)>2$.
 Hence  $M(54)$ is hyperbolic.
 In particular $M(54)$ has two discrete faithful
 characters corresponding to the hyperbolic structure
 which must be contained in $C$ (by the proof of Proposition~\ref{hyperbolic larger than s+2}).
 So these two points of $C$ contribute to the norm $\|(54,1)\|$ by $4$ beyond $s$.
  Since $\D(\a,\b)=3$ and $\b$ is $O$-type,
 $\a$ cannot be $I(3)$-type (cf.  Remark~\ref{qfactor}). Similarly since $\D(\a,\m)=2$, $\a$
 cannot be $I(2)$-type. Thus $\a$ is  $I(5)$-type by Lemma~\ref{q}.
 Hence the two irreducible characters
 of $M(\a)$ factor through $M(54)$ (cf. Remark~\ref{qfactor} again), and
  these two characters are contained in $C$ by Theorem~\ref{norm of finite slope} (6b).
  Hence these two points of $C$ contribute another $4$ to the norm $\|(54,1)\|$  beyond $s$
  (cf. Remark~\ref{zeroes of factors}).

 The  Alexander polynomial
 of $K$ has $4$ simple roots of orders  divisible by $6$ (they are roots
 of the factor $\D_{T(3,2)}(t^3)$) which provide
   $4$ reducible non-abelian characters  which factor
   through $M(54)$.
Let $\chi_{\r_0}$ be the irreducible character of $M(\b)$
such that the image of $\r_0$ is the octahedra group.
By Lemma~\ref{q}, $\r_0$ also factors through $M(54)$.
If all the $4$ reducible non-abelian characters
and the $O$-type character $\chi_{\r_0}$  are contained in $C$, then $\|(54,1)\|\geq s+4+4+8+2=s+18$.
 On the other hand by Lemma~\ref{s+2}, $\|(52,1)\|\leq s+2$ from which we see that
  $\|(54,1)\|\leq 3s+2$.
 So $s+18\leq 3s+2$, i.e. $s\geq 8$, yielding a contradiction with the early conclusion $s\leq 7$.

So some of the $4$ reducible non-abelian characters or
the $O$-type character $\chi_{\r_0}$ is not contained in $C$.
If some of the $4$ reducible non-abelian characters is not contained in $C$ then we can get a contradiction similarly as in Case 2.
Thus we may suppose that the $O$-type character $\chi_{\r_0}$ is not contained in $C$ and all the $4$ reducible non-abelian characters are contained in $C$. Then the same argument as above yields
$\|(54,1)\|\geq s+16$ and $s+16\leq 3s+2$, i.e. $s\geq 7$.
Hence $s=7$.
Since $\chi_{\r_0}$ is not contained in $C$, we have $\|\b\|\leq s+1=8$ by Theorem~\ref{norm of finite slope} (4).
As $\|\a\|=\|(103,2)\|=s+4=11$,
the point $(824/11, 16/11)$ (which lies in the line segment with end points $(0,0)$ and $(103,2)$)
has norm $8$.
So the line segment  with endpoints $(50,1)$ and
$(824/11, 16/11)$ is contained in $B(8)$.
The intersection point of this line segment
with the line passing through $(0,0)$ and $(51,1)$
is $(1224/19, 24/19)$.
So $\|(1224/19, 24/19)\|\leq 8$.
But $\|(1224/19, 24/19)\|=\|\frac{24}{19}(51,1)\|
=\frac{24}{19}\|(51,1)\|\geq \frac{24}{19}s=168/19$.
So we would have $168/19\leq 8$, which is absurd.
This final contradiction shows that $50$ cannot be a finite surgery slope for $K$.

So  $\b$ is possibly $51$ or $52$.
If $51$ is a finite non-cyclic surgery slope for $K$, then it is $T$-type.
But this cannot happen from Table~\ref{table:T} and Lemma~\ref{different A-poly}.
With a similar argument as that of the claim in Case 1, $51$ cannot be a cyclic surgery slope.
So $\b$ is possibly $52$. In fact $52$ is a $D$-type surgery slope
for  $[17,3;3,2]$ (\cite[Table~1]{BH}).
From Lemma~\ref{DOC} (3), $52$ cannot be a cyclic surgery slope.
Hence in this case we have possibly
$\a=103/2$ an $I$-type, $\b=52$ a $D$-type, with $[17,3;3,2]$ as an sample knot.

{\bf  Case 8. $\a=113/2$}.
This case can be handled entirely  as Case 7, and
the only possibility is: $\a=113/2$ an $I$-type, $\b=56$ a $D$-type,
with $[19,3;3,2]$ as a sample knot.

The proof of Theorem~\ref{bound nonintegral} is finished.
%%%%%
%%%%%
%%%%%
%%%%%
%%%%%

\section{Proof of Theorem~\ref{main thm}--Part (II) and proof of Theorem~\ref{prism and det}}\label{sect:main thm-part II}

Theorem~\ref{main thm} is included in the combination of Theorem~\ref{bound nonintegral}
and the following theorem which we prove in this  section.

\begin{thm}\label{thm:integer case}Let $K$ be a hyperbolic knot in $S^3$ which does not admit
half-integer finite surgery.
\newline
(1) The distance between any two integer finite surgery slopes is at most two.
Consequently there are at most three nontrivial finite surgery slopes and
if three, they are consecutive integers.
\newline
(2) There are at most two integer non-cyclic finite surgery slopes
and all possible such pairs of slopes  are:
$$\begin{array}{l}
\{22,23, P(-2,3,9)\}, \{28,29, -K(1,1,0)\},\{50,52, [17,3;3,2]\},\\
\{56,58, [19,3;3,2]\}, \{91, 93, [23,4;3,2]\},\{99,101, [25,4;3,2]\}.\end{array}$$
Also included to each pair is a sample knot which has identical knot Floer homology
and the pair of finite surgeries as $K$.
\newline
(3) If there are three integer finite surgery slopes
on $K$, they must be the triple $(17,18,19)$ and they produce
the same spherical space forms as on the pretzel knot $P(-2,3,7)$.
Also $K$ has the same knot Floer homology as $P(-2,3,7)$.
\newline
(4)  If there are three integer finite surgery slopes
on $K$, then $K$ has to be the knot $P(-2,3,7)$.
\newline
(5) If there are two finite surgery slopes on $K$ realizing distance two,
they must be one of the following pairs:
$$\begin{array}{l}\{17,19, P(-2,3,7)\},
\{21,23,[11,2;3,2]\}, \{27,25,[13,2;3,2]\}, \{37,39,[19,2;5,2]\}, \{43,41,[21,2;5,2]\}
\\\{50,52, [17,3;3,2]\},\{56,58, [19,3;3,2]\}, \{91, 93, [23,4;3,2]\},\{99,101, [25,4;3,2]\}.\end{array}$$
Also included to each pair is a sample knot which has identical knot Floer homology
and the pair of finite surgeries as $K$.
\end{thm}

Of course part (4) of the theorem supersedes part (3), but to get part (4) we need to get part (3)
first.

Since $K$ is assumed to have no half-integer finite surgery slope,
 all nontrivial finite surgery slopes
of $K$ are integers and their mutual distance is at most $3$ by Known Facts~\ref{known1} (2).
Also  Known Facts~\ref{known2} (3) puts significant restrictions
 on possible  $T$-, $O$- and $I$-type  finite surgeries, and
 Known Facts~\ref{known2} (6) on cyclic surgeries.
The main issue is when a $D$-type finite surgery is involved, in which
case our method is to apply the Casson--Walker invariant.
We first make some preparation accordingly. Along the way we shall
also give a proof of Theorem \ref{prism and det}.

Let $P(n,m)$ be the prism manifold with Seifert invariants
\[(-1;(2,1),(2,1),(n,m)),\]
where $n>1$, $\gcd(n,m)=1$.
It is easy to see $P(n,-m)=-P(n,m)$, and $|H_1(P(n,m))|=|4m|$.
As noted earlier, every $D$-type spherical space form is homeomorphic to
some $P(n,m)$.

\begin{figure}[ht]
\begin{picture}(340,100)
\put(140,0){\scalebox{0.7}{\includegraphics*%[0pt,0pt][100pt,100pt]
{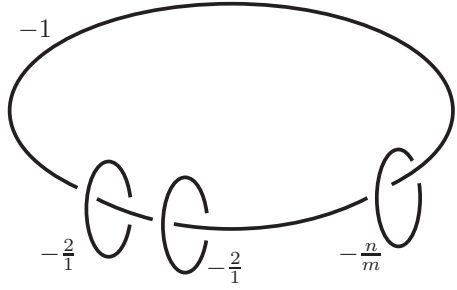}}}

\put(144,90){$-1$}

\put(152,5){$-\frac{2}{1}$}

\put(215,0){$-\frac{2}{1}$}

\put(265,5){$-\frac{n}{m}$}

\end{picture}
\caption{\label{fig:SFS}A surgery diagram of $P(n,m)$}
\end{figure}

 Given a real number $x$, let $\{x\}=x-\lfloor x\rfloor$ be the fractional part of $x$. Given a pair of coprime integers $p,q$ with $p>0$, let $\mathbf s(q,p)$ be the Dedekind sum
\[
\mathbf s(q,p)=\sum_{i=1}^{p-1}\left(\left(\frac ip\right)\right)\left(\left(\frac{iq}p\right)\right),
\]
where
\[
((x))=\left\{
\begin{array}{ll}
\{x\}-\frac12, &\text{if }x\in\mathbb R\setminus\mathbb Z,\\
0, &\text{if }x\in\mathbb Z.
\end{array}
\right.
\]
By \cite{Lescop}, the Casson--Walker invariant of $-P(n,m)$, when $m>0$, can be computed by the formula
\begin{equation}\label{eq:CW}
\lambda(-P(n,m))=-\frac1{12}\left(-\frac nm(\frac1{n^2}-\frac12)-\frac mn+3+12\mathbf s(m,n)\right).
\end{equation}

The  Casson--Walker invariant has the following surgery formula on knots in $S^3$:
\begin{equation}\label{eq:CWsurgery}
\lambda(S^3_K(p/q))=-\mathbf s(q,p)+\frac{q}{p}\Delta_K''(1).
\end{equation}
Here, the Alexander polynomial $\Delta_K(t)$ is normalized to be symmetric and
 $\Delta_K(1)=1$.

Note that the Casson--Walker invariant has the property $\l(-Y)=-\l(Y)$.
In our application, it is sufficient to use only
$|\l(Y)|$. So we do not have to worry about the orientation of the manifold involved.

\begin{lem}\label{D-surgery and det}If  a knot $K$ in $S^3$ admits a $D$-type finite surgery slope $\b$, then
 the meridian coordinate of $\b$ is an integer $4m$  and the resulting prism manifold
 is $\e P(n,m)$ for some $\e\in\{\pm\}$, where  $n$  is the determinant of the knot $K$.
 In particular if $\det(K)=1$, then $K$ does not admit $D$-type finite surgery.
\end{lem}

\pf This lemma is just a refinement of Lemma~\ref{q} (1) and Lemma~\ref{DOC} (1).
 Note that $|H_1(\e P(n,m))|=|4m|$. So we just need to show that $n=\det(K)$.
   \cite[Theorem 10]{K} says that for any knot $K$ in $S^3$, its knot group has precisely
$(\det(K)-1)/2$ distinct $PSL_2(\c)$ dihedral representations, modulo conjugation,
and moreover any such representation will kill any slope
with even meridian coordinate.
On the other hand for a prism manifold $\e P(n,m)$, it has precisely
$(n-1)/2$ distinct $PSL_2(\c)$ dihedral representations, modulo conjugation
(\cite[Proposition D]{AB}).
As  a $D$-type finite surgery on a knot in $S^3$ is actually $D(2)$-type
(cf. Lemma~\ref{q} (1)), the set of dihedral representations of
any $D$-type surgery manifold on a knot in $S^3$ is precisely the set
of dihedral  representations of the knot group.
The conclusion of the lemma follows. \qed

\begin{proof}[{\bf Proof of Theorem~\ref{prism and det}}]
By Lemma~\ref{D-surgery and det}, the surgery slope  is $4m/q$, $(4m,q)=1$,
and $n=\det(K)$. Up to reversing the orientation of $P(n,m)$, we may assume
$m$ is positive and up to replacing $K$ by its mirror image, we may assume
that $q>0$.
So by Known Facts~\ref{known2} (2), $4m/q>2g(K)-1$.
Also by Known Facts~\ref{known2} (2), the nonzero coefficients of the Alexander polynomial of
$K$ are all $\pm 1$ and by Known Facts~\ref{known2} (1), the knot $K$ is fibred
and thus the degree of the Alexander polynomial is $2g(K)$,  which together
imply that \[n=\det(K)=|\D_K(-1)|\leq 2g(K)+1<\frac{4m}{q}+2\leq 4m+2.\]

If $n=4m+1$, then $g(K)=2m$, $q=1$  and
\[
\Delta_K(t)=1+\sum_{i=1}^{2m}(-1)^i(t^{i}+t^{-i}).
\]
It follows that $\Delta''_K(1)=4m^2+2m$. Since
\begin{equation}\label{eq:s14m}
\mathbf s(1,4m)=\sum_{i=1}^{4m-1}\left(\frac{i}{4m}-\frac12\right)^2=\frac{8m^2-6m+1}{24m},
\end{equation}
using (\ref{eq:CWsurgery}), we get
\[
\lambda(S^3_K(4m))=\frac{16m^2+18m-1}{24m}.
\]

On the other hand, we have
\begin{eqnarray*}
\mathbf s(m,4m+1)&=&\sum_{i=1}^{4m}\left(\left(\frac{i}{4m+1}\right)\right)\left(\left(\frac{im}{4m+1}\right)\right)\\
&=&\sum_{i=1}^{4m}\left(\left(\frac{4i}{4m+1}\right)\right)\left(\left(\frac{-i}{4m+1}\right)\right)\\
&=&\sum_{k=0}^3\sum_{j=1}^{m}\left(\left(\frac{4(km+j)}{4m+1}\right)\right)\left(\left(\frac{-km-j}{4m+1}\right)\right)\\
&=&\sum_{k=0}^3\sum_{j=1}^{m}\left(\frac{4j-k}{4m+1}-\frac12\right)\left(\frac{(4-k)m+1-j}{4m+1}-\frac12\right)\\
&=&\frac{4m-m^2}{12m+3}.
\end{eqnarray*}
So it follows from (\ref{eq:CW}) that
\[
\lambda(-P(4m+1,m))=\frac{2m^2-18m+1}{24m}.
\]
Thus $\lambda(S^3_K(4m))\ne\pm\lambda(-P(4m+1,m))$ for any positive integer $m$.
We get a contradiction. This shows $n\ne4m+1$. Since $n=\det(K)$ is odd, we must have $n<4m$.
\end{proof}

We are now ready to prove Theorem~\ref{thm:integer case}.
Suppose $\a$ and $\b$ are two integer  finite surgery slopes on $K$.
To prove part (1) of the theorem we only need to rule out the possibility of $\D(\a,\b)=3$.
So suppose that $\D(\a,\b)=3$. Then one of $\a$ and $\b$, say $\a$, is
an odd integer and the other, $\b$, is an even integer.
We know from Known Facts~\ref{known1} (3) that neither $\a$ nor $\b$
can be $C$-type.
So by Lemma~\ref{q}, $\a$ is of $I$-type or $T$-type and
$\b$ is of $O$-type or $D$-type.
In fact $\a$ cannot be of $T$-type for otherwise
by Lemma~\ref{q} (2) $\a$ is of $T(3)$-type which means, since $\D(\a,\b)=3$,
that  the irreducible representation of $M(\a)$ with image $T_{12}$
also factors through $M(\b)$, i.e.
$M(\b)$ has an irreducible representation with image $T_{12}$.
But this is impossible since any $O$-type or $D$-type spherical space form does not
have such representation.
So $\a$ is of $I$-type.
Now from Tables~\ref{table:O} and~\ref{table:I} one can check quickly
that there is  no sample knot which admits an integer $I$-type surgery and an integer $O$-type surgery, distance $3$ apart (one just need to check for those sample  knots in Table~\ref{table:I} with $\det(K)=3$
and there are only $7$ of them).
Hence by Lemma~\ref{different A-poly} and Known Facts~\ref{known2} (3)
there is no knot in $S^3$  which admits an integer $I$-type surgery and an
integer $O$-type surgery, distance $3$ apart.
So $\b$ is a $D$-type slope.

So we have $\a$ is of $I$-type, $\b$ is of $D$-type and they are distance $3$ apart.
To rule out this case, we shall apply the Casson--Walker invariant.
By Known Facts~\ref{known2} (3), $\a$ is one of the slopes
given in Table~\ref{table:I} and $K$ has the same
Alexander polynomial (in particular the same determinant) as the corresponding sample knot.
We only need to consider those slopes in the table
whose sample knots have determinant larger than $1$.
We  may express such a slope as  $\a=4m+3$ or $4m-3$ for some integer $m>0$.
So we just need to show  that $S^3_K(4m)$ is not a prism manifold.
To do this, we compute $\lambda(S^3_K(4m))$ using (\ref{eq:CWsurgery}) and
compute $\lambda(-P(n,m))$ for $n=\det(K)$ using (\ref{eq:CW}),
and check whether $|\lambda(S^3_K(4m))|$ is equal to $|\lambda(-P(n,m))|$.
Also note that by Known Facts~\ref{known2} (5)
when $4m\le32$, $K$ must  have the same Alexander polynomial as  a corresponding sample knot given in Table~\ref{table:D}.
This finite process of computation shows that the only possible case is when $m=1$ and the corresponding sample knot
is $T(3,2)$. It follows from \cite{OSz3141} that $K$ must be $T(3,2)$, contradicting the assumption that $K$ is hyperbolic.
Part (1) of the theorem is proved.

Part (2) is treated  with a similar  strategy.
Suppose $\a$ and $\b$ are two distinct integer non-cyclic finite surgery slopes of $K$.
We are going to show that $(\a,\b)$ must be one of the pairs listed in part (3) with the
corresponding sample knot playing the said role, and that there cannot be the third non-cyclic
finite surgery on $K$.
By part (1), $\D(\a,\b)=2$ or $1$.

Let us first consider the case when $\D(\a,\b)=2$.
Then $\a$ and $\b$ are both odd or both even integers.
If they are both odd, then each of $\a$ and $\b$ is of $T$-type or $I$-type by Lemma~\ref{q}.
Then by Known Facts~\ref{known2} (3) and Lemma~\ref{different A-poly}, we just need to check
which sample knots in Table~\ref{table:T} and Table~\ref{table:I}  have
two slopes listed in  these tables distance two apart.
There are only three such instances:
$$\{1,3, T(3,2)\}, \{91, 93, [23,4;3,2]\}, \{99,101, [25,4;3,2]\}.$$
The first instance can be excluded due to \cite{OSz3141}.
In the second instance, we just need to show that $92$ cannot be a finite non-cyclic surgery
slope for the same knot $K$. Suppose otherwise that $92$ is a non-cyclic finite surgery slope for $K$.
Then it must be a $D$-type surgery slope by Lemma~\ref{q}. But by Lemma~\ref{D-surgery and det},
 the resulting prism manifold would be $\e P(23,23)$, which does not make sense
  since $23$ and $23$ are not relative prime integers.
Thus $92$ cannot be a $D$-type slope for $K$.
The third instance can be treated exactly as in the second one.

If both $\a$ and $\b$  are even, then each of $\a$ and $\b$ is of $O$-type or $D$-type by Lemma~\ref{q}.
Note that  $\a$ and $\b$ cannot both be $D$-type by Known Facts~\ref{known1} (4).
From Known Facts~\ref{known2} (3) and Table~\ref{table:O},
 we see that $\a$ and $\b$ cannot both be $O$-type.
 So we may assume that $\a$ is an $O$-type slope and $\b$  a $D$-type slope.
Each slope $\a$ in Table~\ref{table:O} can be expressed as $4m+2$. If
$4m$ or $4m+4$ is a $D$-type slope for $K$, then we should have
 $|\l(P(3,m))|=|S^3_K(4m)|$ or $|\l(P(3,m+1))|=|S^3_K(4m+4)|$ respectively.
 Calculation using (\ref{eq:CW}) and (\ref{eq:CWsurgery})
shows that this happens only in
three instances: $$\{2,4, T(3,2)\}, \{50,52, [17,3;3,2]\}, \{56,58, [19,3;3,2]\}.$$
Again the first instance cannot happen for a hyperbolic knot due to \cite{OSz3141}.
In the second instance  $52$ is indeed a $D$-type slope for $[17,3;3,2]$, and one can easily
rule out the possibility for $51$ to be a non-cyclic finite surgery slope.
The third instance can be treated exactly as in the second one.

Next we  consider the case when $\D(\a,\b)=1$.
As in part (1), we may assume $\a$ is a $T$-type or $I$-type slope
and $\b$ an $O$-type or $D$-type slope.
With a similar process as used in part (1),
 we only obtain the following instances (with the trefoil case  excluded):
  $$\{7, 8, T(5,2)\}, \{12, 13, T(5,2)\},
 \{28,29, -K(1,1,0)\},\{22,23, P(-2,3,9)\}.$$
We note that each pair of slopes do realize on the sample
knot as non-cyclic finite surgery slopes
and by the result obtained in the preceding two  paragraphs, there is no third non-cyclic finite surgery
in each  instance for the same hyperbolic knot $K$.
The first two instances with the sample knot $T(5,2)$ can be rule out 
by Theorem \ref{thm:DiSurg}.
Part (2) of the theorem is proved.

To prove part (3), suppose that $K$ has three integer finite surgery slopes.
They are consecutive integers by part (1).
At least one of them  is non-cyclic by Known Facts~\ref{known1} (1)
and at most two of them are non-cyclic by part (2).

If two of them are non-cyclic, then
the two slopes must be one the pairs listed in
 part (2) of the theorem with the corresponding sample knot.
The case of  $\{7, 8, T(5,2)\}$ cannot happen since any hyperbolic knot
cannot have a cyclic surgery slope $6$ or $9$ by
Known Facts~\ref{known2} (7).
In each of other cases the same hyperbolic knot $K$ can no longer have a cyclic
surgery slope.
For if it does, then by Known Facts~\ref{known2} (6) there will be a Berge knot
having the same cyclic slope and same Alexander polynomial
as the corresponding sample knot attached to the pair of non-cyclic
finite surgery slopes.
But one can check (which is a finite process) that there does not exist
such Berge knot.

So we may assume that there is exactly one non-cyclic finite slope and two
cyclic surgery slopes on  $K$.
By Known Facts~\ref{known1} (4), the non-cyclic finite surgery slope $\a$
cannot be $D$-type or $O$-type.
So $\a$ is a $T$-type or $I$-type surgery slope belonging to
Table~\ref{table:T} or Table~\ref{table:I}.
In particular it is a positive integer less than or equal to $221$.
Also the determinant of $K$ is one by Lemma~\ref{DOC} (3)
and so there are only $14$ possible value for $\a$.
In each of the $14$ cases we check
that there is no Berge knot $K_0$ such that $K_0$ admits two integer
cyclic surgery slopes which form consecutive integers with $\a$
and that $K_0$ has the same Alexander polynomial as the sample knot attached to
the slope $\a$
in Table~\ref{table:T} or Table~\ref{table:I}, except
the case when $\a=17$.
In fact we have

\begin{lem}
If for some positive integer $p\leq 222$, $p$ and $p+1$ are cyclic surgery slopes for a
nontrivial Berge Knot
 then $p$ is one of the $10$ values: 18, 30, 31, 67, 79, 116, 128, 165, 177, 214.
 If for such $p$,   $p$, $p+1$ and a slope $\a$ from  Table~\ref{table:T} or Table~\ref{table:I}
form consecutive integers, then the  corresponding Berge knot and the corresponding sample knot associated to $\a$
have different Alexander polynomials, except when $\a=17$ ($p=18$).
\end{lem}

\pf
We use a Mathematica program to check the following fact: if for some positive integer $p\leq222$ and integers $q_1,q_2$, both $L(p,q_1)$ and $L(p+1,q_2)$ satisfy Condition~\ref{cond:Corr}, then $p$ is one of the $10$ values in the lemma. For each of these values of $p$, there is only one possible Alexander polynomial for the corresponding knot, which can be realized by a Berge knot. In fact, for each integer $n$, the Eudave-Mu\~noz knot $k(2,2,n,0)$ has two lens space surgeries with slopes $49n-18$
and $49n-19$, and the knot $k(2,-1,2,0)$ has two  lens space surgeries with slopes $30,31$ \cite{EM1}.
Moreover, the Alexander polynomials of these  knots are different from those  of knots in Table~\ref{table:T} or Table~\ref{table:I} except when $p=18$.
\qed

When $\a=17$, the sample knot is $P(-2,3,7)$ which does have
$18$ and $19$ as cyclic surgery slopes.
Part (3) of the theorem is proved.

Based on part (3), we can quickly prove part (4).
If $K$ admits three nontrivial finite surgery, then by part (3) the surgery slopes
are the triple $17,18,19$, with $17$ an $I$-type and $18$, $19$ $C$-type slopes,
 and $K$ has the same knot Floer homology as  $P(-2,3,7)$.
 In particular $K$ has genus $5$.
 Now applying Known Facts \ref{known2} (7) to the cyclic slope $19$,
 we see that $K$ is a Berge knot.
But among all hyperbolic Berge knots, $P(-2,3,7)$ is the only one
which admits cyclic slope $18$ or $19$. This last assertion follows from
Lemma 1, Theorem 3 and Table of Lens Spaces of \cite{Berge}.

To prove part (5), assume that
$\a$ and $\b$ are two integer finite surgery slopes for $K$ with $\D(\a,\b)=2$.
If both $\a$ and $\b$ are non-cyclic, then by part (2)
they are one of the pairs
$$\{50,52, [17,3;3,2]\}, \{56,58, [19,3;3,2]\},\{91, 93, [23,4;3,2]\}, \{99,101, [25,4;3,2]\}.$$
So we may assume
that exactly one of them, say $\a$, is non-cyclic  by Known Facts~\ref{known1} (1)
and $\a$ must be a $T$-type or $I$-type slope
 by Known Facts~\ref{known1} (4) (5).
So  $\a$
is one of the slopes  in Table~\ref{table:T} or Table~\ref{table:I} and $K$ has
the same Alexander polynomial as the corresponding sample knot associated  to $\a$.
Again in this situation we only need to check the following lemma.

\begin{lem}
If there exists a hyperbolic knot $K$
such that $K$ admits a cyclic surgery slope $p$, $p\leq 223$,
which is distance two from a slope $\a$ in Table~\ref{table:T} or Table~\ref{table:I}
and that $K$ has the same Alexander polynomial
as the sample knot attached to $\a$ in Table~\ref{table:T} or Table~\ref{table:I},
then $\a,p$ are one of the pairs $$\{17,19, P(-2,3,7)\},
\{21,23,[11,2;3,2]\}, \{27,25,[13,2;3,2]\}, \{37,39,[19,2;5,2]\}, \{43,41,[21,2;5,2]\}$$
and each pair is realized on the attached sample knot.
\end{lem}
\pf
Again, this is proved by using a Mathematica program to check Condition~\ref{cond:Corr} for
each $\a$-surgery in Table~\ref{table:T} or Table~\ref{table:I} and a lens space $L(p,q)$
with $\D(\a,p)=2$.
We get all such pairs $\a,p$ satisfying Condition~\ref{cond:Corr} along with the recovered
Alexander polynomials,  which yield
corresponding sample knots in Table~\ref{table:T} or Table~\ref{table:I}.
Such sample knot  is either a torus knot ($T(3,2)$ or $T(5,2)$) or an iterated torus knot listed
in the lemma or $P(-2,3,7)$. The case of torus knots can be ruled out by \cite{OSz3141} and Known Facts \ref{known2} (7).
\qed

Part (5) of the theorem is proved.

%%%%%
%%%%%
%%%%%
%%%%%
%%%%%

\section{Proof of Theorem~\ref{thm:DiSurg}}\label{sect:DiSurg}

We first consider the case of Theorem~\ref{thm:DiSurg} when $K$ is non-hyperbolic. This case follows easily from existing results.
In fact, by Thurston's Geometrization theorem, if $K$ is not hyperbolic, then $K$ is either a torus knot or a satellite knot.
The classification of surgeries on torus knots is carried out in \cite{Moser}. By
\cite[Corollary~1.4]{BZ1}, if a satellite knot admits a finite surgery, then this knot is a cable of a torus knot. Finite surgeries on such cable knots are classified in \cite[Theorem~7]{BH}. From these classification results, one can readily check that $T(2m+1,2)$ is the only non-hyperbolic knot in $S^3$ admitting a surgery to $\varepsilon S^3_{T(2m+1,2)}(4n)$ with
slope $4n$ (easy to see this among torus knots, and on cables of torus knots see \cite[Table~1]{BH}).

From now on, we assume $K$ is hyperbolic.
We first get an estimate on the genus of $K$ applying the correction terms from  Heegaard Floer homology.
In  our current situation
 the correction terms of $S^3_{K}(4n)$ are given by the formula
\begin{equation}\label{eq:CorrSurg}
d(S^3_{K}(4n),i)=-\frac14+\frac{(2n-i)^2}{4n}-2t_{\min\{i,4n-i\}}(K), \;\;i=0,1,2,\cdots, 4n-1.
\end{equation}

For the torus knot $T(2m+1,2)$, the coefficients of its normalized Alexander polynomial are
$$a_i=\left\{
\begin{array}{ll}
(-1)^{m-i}, &\text{if }|i|\le m,\\
0, &\text{otherwise.}
\end{array}
\right.$$
So
\begin{equation}
\begin{array}{ll}t_i(T(2m+1,2))&=\left\{\begin{array}{ll}\sum_{j=1}^{m-i}(-1)^{m-i-j}j, &\text{for $0\leq i<m$}\\
0, &\text{for $i\geq m$}\end{array}\right.\\&
=\left\{\begin{array}{ll}\lceil\frac{m-i}2\rceil, &\text{for $0\leq i<m$}\\
0, &\text{for $i\geq m$.}\end{array}\right.\end{array}
\end{equation}

For any $k\in\mathbb Z$, let $\theta(k)\in\{0,1\}$ be the reduction of $k$ modulo $2$. Let $\zeta=n-m\in\{0,1\}$.
Applying (\ref{eq:CorrSurg}) to $T(2m+1,2)$, we can compute that
\begin{equation}\label{eq:CorrT}
d(S^3_{T(2m+1,2)}(4n),i)=\left\{
\begin{array}{ll}
-\frac14+\frac{i^2}{4n}+\zeta-\theta(n-\zeta-i),&\text{if }0\le i< n,\\
-\frac14+\frac{(2n-i)^2}{4n}, &\text{if }n\le i\le 2n,\\
d(S^3_{T(2m+1,2)}(4n),4n-i), &\text{if }2n< i< 4n.
\end{array}
\right.
\end{equation}

\begin{prop}\label{prop:GenusBound}
If $S^3_K(4n)\cong \e S^3_{T(2m+1,2)}(4n)$, then $\varepsilon=+$, and $g(K)$, the Seifert genus of $K$, is less than or equal to $n$.
\end{prop}
\begin{proof}[{\bf Proof}]
Since $S^3_K(4n)\cong \e S^3_{T(2m+1,2)}(4n)$, there exists an affine isomorphism $\phi\co\mathbb Z/4n\mathbb Z\to\mathbb Z/4n\mathbb Z$ such that
\begin{equation}\label{eq:AffCorr}
d(S^3_K(4n),i)=\varepsilon d(S^3_{T(2m+1,2)}(4n),\phi(i)).
\end{equation}
The map $\phi$ sends the spin spin$^c$ structures of $S^3_K(4n)$ to the spin  spin$^c$ structures of $S^3_{T(2m+1,2)}(4n)$, namely, $\phi(\{0,2n\})=\{0,2n\}$.
 Using (\ref{eq:AffCorr}), (\ref{eq:CorrSurg}) and (\ref{eq:CorrT}) for $i=0$ and $i=2n$,
 one easily sees that $\e =+$ since otherwise  $t_{2n}(K)$ would be non-integer valued.
 Since $\phi$ is an affine isomorphism of $\mathbb Z/4n\mathbb Z$, $\phi$ must send $\{n,3n\}$ to $\{n,3n\}$.
 Using (\ref{eq:AffCorr}), (\ref{eq:CorrSurg}) and (\ref{eq:CorrT}) for $i=n$ and $i=3n$,
  together with $\e=+$, one sees directly  that $t_n(K)=0$.
It follows from (\ref{eq:tsProperties}) that $g(K)\le n$.
\end{proof}

\begin{cor}\label{punctured Klein bottle}
Suppose that $S^3_K(4n)\cong \e S^3_{T(2m+1,2)}(4n)$ and that $K$ is hyperbolic, then $n=m=g(K)$ and $K$ has the same Alexander polynomial as $T(2m+1,2)$. Moreover,
there is a once-punctured Klein bottle properly embedded in the exterior of $K$ of boundary slope $4m$.
\end{cor}
\begin{proof}[{\bf Proof}]
By Proposition~\ref{prop:GenusBound}, $g(K)\le n$.
Since $S^3_K(4n)$, being a prism space, contains a Klein bottle,
it follows from \cite[Corollary~1.3]{IT} that  $4n\le 4g(K)$. So $n=g(K)$.
Again by \cite[Corollary~1.3]{IT},  $4n$  is the boundary slope of a once-punctured Klein bottle in the exterior of $K$.

Next we prove that $n=m$. Otherwise, we have $n=m+1$.
By Known Facts~\ref{known2}~(2),
\[\Delta_K(t)=(-1)^r+\sum_{i=1}^r(-1)^{i-1}(t^{n_i}+t^{-n_i}),\]
where
\[m+1=n_1>n_2>\cdots>n_r>0.\]
Lemma~\ref{D-surgery and det} implies that $|\Delta_K(-1)|=\det(K)=\det(T(2m+1,2))=2m+1$, so $r=m$ or $m+1$.

If $r=m$, since $|\Delta_K(-1)|=\det(K)=2m+1$, we see that $n_i+i-1$ has the same parity as $m$ for any $i\in\{1,2,\dots,m\}$. This contradicts the assumption that $n_1=m+1$.

If $r=m+1$, then $\Delta_K(t)=(-1)^{m+1}+\sum_{i=0}^m(-1)^i(t^{m+1-i}+t^{-m-1+i})$ and $\det(K)=2m+3\ne2m+1$, we also get a contradiction.

So we have proved $n=m$. Since $|\Delta_K(-1)|=\det(K)=2m+1$, $\Delta_K(t)$ has to be $\Delta_{T(2m+1,2)}(t)$.
\end{proof}

Let $M$ be the exterior of $K$.
Let $P$ be a once-punctured Klein bottle  in $M$ with boundary slope $4m$, provided by
Corollary~\ref{punctured Klein bottle}. Let $H$ be a regular neighborhood of $P$ in $M$, then $H$ is a handlebody of genus $2$.
Let $H'=M\setminus H$.
Then $F=H\cap H'=\p H\cap \p H'$
 is a twice-punctured genus one surface properly embedded in $M$.
Each component of  $\p F$ is a simple closed curve in $\p M$ parallel
to $\p P$ and thus is of slope $4m$.
Note  that $\p F$ separates  $\p M$ into two annuli
$A$ and $A'$ such that $\p H=F\cup A$ and $\p H'=F\cup A'$.

\begin{lem}\label{lem:Compress}
$F$ is compressible in $H'$ and is incompressible in $H$.
\end{lem}
\begin{proof}[{\bf Proof}]
Let $Q$ be the genus $g(K)=m$ Seifert surface for $K$, provided by
Corollary~\ref{punctured Klein bottle}.
By the incompressibility of the surfaces $P$ and $Q$, we may  assume that
$P$ and $Q$ intersect transversely,  that $P\cap Q$ contains no circle component
which bounds a disk in $P$ or $Q$, and that $\p P$ intersects $\p Q$ in exactly
$4m$ points.
Hence $P\cap Q$ has precisely  $2m$ arc components each of which is
essential in $P$ and $Q$ (again because the incompressibility of $P$ and $Q$).

Now consider the intersection graphs $G_P$ and $G_Q$ determined by the surfaces
 $P$ and $Q$ as usual (see, e.g. \cite{IT}), that is,
  if $\widehat P$ (resp. $\widehat Q$) is the closed surface
  in $M(4m)$ (resp. in $M(0)$)  obtained from $P$ (resp. $Q$) by capping off its boundary by a disk,
  then $G_P$ (resp. $G_Q$) is a graph in $\widehat P$  (resp. $\widehat Q$)
 obtained by taking the disk $\widehat P\setminus P$ (resp.
 $\widehat Q\setminus Q$) as a fat vertex and taking the
 arc components of $P\cap Q$ as edges.
 In particular each $G_P$ and $G_Q$ has precisely $2m$ edges.

A simple Euler characteristic calculation shows that the graph $G_Q$ must have
at least one disk face $D$.
Let $D'=D\setminus H=D\cap H'$, then $D'$ is a properly embedded disk in $H'$.

We claim that $\partial D'$ is an essential curve on $\p H'$. In fact, a component $C$ of $\p F$ is  an essential curve on $\p H'$ and is also an essential curve in $\p M$ of slope $4m$. As  $C$ has $4m$ intersection points with $\partial Q$, all with the same sign,
 if $D$ is a $k$--gon face of the graph $G_Q$, then $\partial D'$ has $k$ intersection points with $C$, all with the same sign. So $\partial D'$ is an essential curve on $\p H'$.

The claim proved in the last paragraph implies that $D'$ is a compressing disk for $\p H'$.

If $F$ is incompressible in $H'$, then
by the handle addition lemma \cite{Jaco}, the manifold
obtained by attaching a 2-handle to $H'$ along $A'$ will
give a manifold $Y'$ with incompressible boundary (which is a torus).
The manifold
$Y$ obtained by attaching a 2-handle to $H$ along $A$ gives
  a twisted $I$--bundle over Klein bottle, whose boundary is incompressible.
Then $M(4m)$ is the union of $Y$ and $Y'$ along their torus boundary
and thus is a Haken manifold, a contradiction.
Therefore $F$ is compressible in $H'$.

Note that $H$ is an $I$-bundle over $P$ and $F$ is
 the horizontal boundary of $H$ with respect to the $I$-bundle structure.
 It follows that the composition of the inclusion map $F\hookrightarrow H$
 and the projection map $H\to P$ with respect to the $I$-bundle structure
 is an $2$-fold covering map and is thus $\pi_1$--injective.
 As the fundamental group of $H$ is isomorphic to $\pi_1(P)$, it  follows that
$F$ is $\pi_1$-injective in $H$ and thus is incompressible in $H$.
\end{proof}

\begin{lem}\label{lem:Arc}
Let $\widehat P\subset M(4m)$ be the Klein bottle obtained by capping off $\partial P$ with a disk, and let $\nu(\widehat P)$ be its tubular neighborhood. Then $\nu(\widehat P)$ is a twisted $I$-bundle over $\widehat P$, $V=M(4m)\setminus\nu(\widehat P)$ is a solid torus, and the dual knot $K'\subset M(4m)$ can be arranged by an isotopy to intersect  $\nu(\widehat P)$ in an $I$-fibre and intersect $V$ in a boundary parallel arc.
\end{lem}
\begin{proof}[{\bf Proof}]
By Lemma~\ref{lem:Compress},  $F=\p H'\setminus A'$ is compressible in $H'$.
Let $D_*$ be a compressing disk for $F$ in $H'$.
Let $\widehat F$ be the closed surface in $M(4m)$ obtained from $F$ by capping off
each component of $\p F$ with a disk. Then $\widehat F$ is a torus.

If $\p D_*$ is an inessential curve  in the torus $\widehat F$, i.e. bounds a disk $B$ in
$\widehat F$, then $B$ must contain both components of  $\p F$ since $\p D_*$ is an essential curve in $F$ and
$\p M$ is incompressible in $M$. So  compressing $F$ with $D_*$ produces the
disjoint union of a torus $T_*$ and an annulus $A_*$.
As $M$ is hyperbolic, the torus $T_*$  bounds a solid torus
in $M$ or is parallel to $\p M$ in $M$.
From the construction of $T_*$ we see that $T_*$ cannot be parallel to $\p M$
since otherwise $P$ would be contained in the regular neighborhood of $\p M$
bounded by $\p M$ and $T_*$, which is obviously impossible.
So $T_*$ bounds a solid torus in $M$ and in fact in $H'$.
Similarly the annulus $A_*$ cannot be essential in $M$ and thus must be parallel
to $\p M$. In fact $A_*$ must be parallel to $A'$ in $H'$. In particular, there exists a proper disk $D'\subset H'$ whose boundary consist of an essential arc in $A_*$ and an essential arc in $A'$.
It follows that the surface $S$ which is $\p H$ pushed slightly into the interior of $M$
and $\p M$ bound a compression body. In other words, $S$ is a genus two Heegaard surface of $M$.
Attaching a $2$-handle to $H'$ along $A'$ will cancel the $1$-handle with cocore $D'$, hence we get a solid torus $V$.

If $\p D_*$ is an essential curve in the torus $\widehat F$,
then  compressing $F$ with $D_*$ gives
 an annulus $A_{\#}$.
Again as $M$ is hyperbolic, $A_{\#}$  must be parallel
to $\p M$ and in fact must be parallel to $A'$ in $H'$.
This implies  that the surface $S$ which is $\p H$ pushed slightly into the interior of $M$
and $\p M$ bound a compression body and thus is  a genus two Heegaard surface of $M$. Let
$D'\subset H'$ be a proper disk whose boundary consist of an essential arc in $A_{\#}$ and an essential arc in $A'$.
Attaching a $2$-handle to $H'$ along $A'$ will cancel the $1$-handle with cocore $D'$, hence we get
a solid torus $V$.

In any case we have shown that $M(4m)$ is the union of $\nu(\widehat P)$ and a solid torus $V$. Some  neighborhood of $K'\cap V$ is the $2$--handle added to $A'$, thus $D'$ gives a parallelism between $K'\cap V$ and an arc in $\partial V$. Some neighborhood of $K'\cap \nu(\widehat P)$ is the $2$--handle added to $A$ consisting of $I$-fibres of $\nu(\widehat P)$. Clearly $K'\cap \nu(\widehat P)$ can be considered as 
an $I$-fibre of $\nu(\widehat P)$.
\end{proof}

\begin{lem}\label{lem:1Bridge}
Let $Z$ be the double branched cover of $S^3$ with ramification locus $K$, and let $\widetilde K\subset Z$ be the preimage of $K$. Then $Z_{\widetilde K}(2m)$ is a $2$--fold cover of $S^3_K(4m)$, which is a lens space. Moreover, let $\widetilde K'\subset Z_{\widetilde K}(2m)$ be the dual knot of $\widetilde K$, then $\widetilde K'$ is a $1$--bridge knot with respect to the standard genus $1$ Heegaard splitting of $Z_{\widetilde K}(2m)$.
\end{lem}
\begin{proof}[{\bf Proof}]
Let $\pi\co Z\to S^3$ be the branched covering map, then $\pi\co Z\setminus \widetilde K\to S^3\setminus K$ is an unramified $2$--fold covering map, and $\pi$ maps the simple loop with slope $2m$ on $\partial\nu(\widetilde K)$ homeomorphically to the simple loop with slope $4m$ on $\partial\nu(K)$. Thus $\pi\co Z\setminus \widetilde K\to S^3\setminus K$ can be extended to an unramified $2$--fold covering map $Z_{\widetilde K}(2m)\to S^3_K(4m)$.

Now we look at the double cover of $S^3_K(4m)$. Since $H_1(S^3_K(4m);\mathbb Z/2\mathbb Z)\cong\mathbb Z/2\mathbb Z$, this cover is unique.
Let $U=\nu(\widehat P)$ be a twisted $I$--bundle over $\widehat P$. Then $S^3_K(4m)$ is the union of $U$ and $V$, where $V$ is the solid torus in Lemma~\ref{lem:Arc}. Let $\pi_U\co\widetilde U\to U$ be the $2$--fold covering map induced by the covering map $\partial U\to \widehat P$. Clearly, $\widetilde U$ is homeomorphic to $T^2\times I$. So we may construct a cover of $S^3_K(4m)$ by gluing two copies of $V$ to $\partial\widetilde U$. As a result, $Z_{\widetilde K}(2m)$ is a lens space. Since $\widetilde K'$ is the preimage of $K'$ under the covering map, it follows from Lemma~\ref{lem:Arc} that $\widetilde K'$ is a $1$--bridge knot.
\end{proof}

\begin{lem}\label{lem:Lspace}
$Z$ is an L-space.
\end{lem}
\begin{proof}[{\bf Proof}]
This follows from a standard fact in Heegaard Floer homology. Notice that $m$ is also the genus of $\widetilde K$. By
\cite[Corollary~4.2 and Remark~4.3]{OSzKnot}, for any Spin$^c$ structure $\mathfrak s$ over $Z$, there exists a Spin$^c$ structure $[\mathfrak s,k]$ over $Z_{\widetilde K}(2m)$, such that $\widehat{HF}(Z,\mathfrak s)\cong \widehat{HF}(Z_{\widetilde K}(2m),[\mathfrak s,k])$. Since $Z_{\widetilde K}(2m)$ is an L-space, we have $\widehat{HF}(Z,\mathfrak s)\cong\mathbb Z$ for every $\mathfrak s$, so $Z$ is also an L-space.
\end{proof}

Now we will use a result due to Hedden \cite{HedBerge} and Rasmussen \cite{RasBerge}. We will use the form in \cite[Theorem~1.4~(2)]{HedBerge}. Although the original statement is only for knots in $S^3$, the same proof works for null-homologous knots in L-spaces.

\begin{thm}[Hedden, Rasmussen]\label{thm:DualFSimple}
Let $Z_1$ be an L-space, $L\subset Z_1$ be a null-homologous knot with genus $g$. Suppose that the $p$--surgery on $L$ yields an L-space $Z_2$, and $p\ge2g$. Let $L'\subset Z_2$ be the dual knot, then $L'$ is Floer simple. Namely, $\mathrm{rank}\widehat{HFK}(Z_2,L')=\mathrm{rank}\widehat{HF}(Z_2)$.
\end{thm}

Let $L(p,q)$ be a lens space. Let $V_1\cup V_2$ be the standard genus $1$ Heegaard splitting of $L(p,q)$, and let $D_i\subset V_i$ be a meridian disk such that $D_1\cap D_2$ consists of exactly $p$ points.
A knot $L$ in a lens space $L(p,q)$ is {\it simple} if it is the union of two arcs $a_1,a_2$, where $a_i$ is a boundary parallel arc in $V_i$ that is disjoint from $D_i$, $i=1,2$. In each homology class in $H_1(L(p,q))$, there exists a unique (up to isotopy) oriented simple knot.

\begin{cor}\label{cor:DualSimple}
$\widetilde K'$ is a simple knot in the lens space $Z_{\widetilde K}(2m)$.
\end{cor}
\begin{proof}[{\bf Proof}]
Lemmas~\ref{lem:1Bridge} and \ref{lem:Lspace} and Theorem~\ref{thm:DualFSimple} imply that $\widetilde K'$ is Floer simple in $Z_{\widetilde K}(2m)$, Lemma~\ref{lem:1Bridge} also tells us that $\widetilde K'$ is $1$--bridge.
Using \cite[Proposition~3.3]{HedBerge}, we see that $\widetilde K'$ is simple.
\end{proof}

\begin{lem}\label{lem:SameHomo}
Let $T'$ be the knot dual to $T=T(2m+1,2)$ in $S^3_{T}(4m)\cong S^3_{K}(4m)$, $\widetilde T'$ be its preimage in $Z_{\widetilde K}(2m)$. Then
\newline
(1) $[K']=\pm[T']$ or $\pm(2m-1)[T']$ in $H_1(S^3_{K}(4m))$.
\newline
(2) $[\widetilde K']=\pm[\widetilde T']$ in $H_1(Z_{\widetilde K}(2m))$.
\end{lem}
\begin{proof}[{\bf Proof}]
(1) Recall that $\mathrm{Spin}^c(S^3_K(4m))$ is an affine space over $H^2(S^3_K(4m))$. In other words, $H_1(S^3_K(4m))\cong H^2(S^3_K(4m))$ acts on $\mathrm{Spin}^c(S^3_K(4m))$. There is a standard way to identify $\mathrm{Spin}^c(S^3_K(4m))$ with $\mathbb Z/4m\mathbb Z$ in \cite[Section~4]{OSzKnot}: Let $W'_{4m}$ be the two-handle cobordism from $S^3_K(4m)$ to $S^3$, let $G$ be a Seifert surface for $K$ and let $\widehat G\subset W'_{4m}$ be obtained by capping off $\partial G$ with a disk. For any integer $i$, let $\mathfrak t_i\in\mathrm{Spin}^c(W'_{4m})$ be the unique Spin$^c$ structure satisfying
\begin{equation}\label{eq:c1Identify}
\langle c_1(\mathfrak t_i),[\widehat G]\rangle=2i-4m.
\end{equation}
Then we have an affine isomorphism $\sigma\co\mathrm{Spin}^c(S^3_K(4m))\to\mathbb Z/4m\mathbb Z$ which sends $\mathfrak t_i|_{S^3_K(4m)}$ to $i \pmod{4m}$.

Let $\mu$ be the meridian of $K$, then $\mu$ is isotopic to $K'$ in $S^3_K(4m)$. Using (\ref{eq:c1Identify}), we see that
\[\sigma(\mathfrak s+\mathrm{PD}[\mu])-\sigma(\mathfrak s)=[\mu]\cdot[G]=1.\]
So the action of $[K']$ on $\mathrm{Spin}^c(S^3_K(4m))$ is equivalent to adding 1 in $\mathbb Z/4m\mathbb Z$. There is a similar result when we replace $K$ with $T$.

We identify $S^3_{K}(4m)$ with $S^3_{T}(4m)$ by a
homeomorphism  $f\co S^3_{K}(4m)\to S^3_{T}(4m)$, which induces a symmetric affine isomorphism $\phi\co \mathrm{Spin}^c(S^3_{K}(4m))\to \mathrm{Spin}^c(S^3_T(4m))$. Clearly $\phi$ is equivariant with respect to the $H_1(S^3_K(4m))=H_1(S^3_T(4m))$ action, where we identify $H_1(S^3_K(4m))$ with $H_1(S^3_T(4m))$ using $f_*$.
If $\phi(i)=ai+b$, consider the actions of $[K']$ and $[T']$ on the Spin$^c$ structures, we get that $[K']=f_*([K'])$ acts as adding $a$ on $\mathrm{Spin}^c(S^3_T(4m))$. Since $[T']$ acts as adding $1$ on $\mathrm{Spin}^c(S^3_T(4m))$, $[K']=a[T']$.

We should have
\[
d(S^3_{T}(4m),i)=d(S^3_{K}(4m),i)=d(S^3_{T}(4m),\phi(i))
\]
for any $i\in\mathbb Z/4m\mathbb Z$,
where the first equality holds since $\Delta_K(t)=\Delta_T(t)$.
We will use (\ref{eq:CorrT}) to compute $d(S^3_{T}(4m),i)$. Note that $m=n$, $\zeta=0$.
Recall from the proof of Proposition~\ref{prop:GenusBound} that $\phi(\{0,2n\})=\{0,2n\}$.

When $m$ is even, it is straightforward to check that the minimal value of $d(S^3_{T}(4m),i)$ is $-\frac14+\frac1{4m}-1$, which is attained if and only if $i=1$ or $4m-1$. So $\phi(1)=1$ or $4m-1$. Since $\phi(0)=0$ or $2m$, $a=\phi(1)-\phi(0)\in\{\pm1,\pm(2m-1)\}\pmod{4m}$.

When $m$ is odd, $d(S^3_{T}(4m),0)\ne d(S^3_{T}(4m),2m)$, so we must have $\phi(0)=0$. We have $d(S^3_{T}(4m),1)=-\frac14+\frac1{4m}$. Since $m$ is odd, $4m+1\equiv5\pmod8$, so $-\frac14+\frac{i^2}{4m}-1\ne d(S^3_{T}(4m),1)$ for any integer $i$. It follows that $d(S^3_{T}(4m),i)=d(S^3_{T}(4m),1)$ only when $i\in\{1,2m\pm1,4m-1\}$. Hence $a=\phi(1)-\phi(0)=\phi(1)\in\{\pm1,\pm(2m-1)\}\pmod{4m}$.

In any case, we proved that $a\in\{\pm1,\pm(2m-1)\}\pmod{4m}$, thus our conclusion holds.

(2) Let \[\tau_*\co H_1(S^3_{K}(4m))\cong\mathbb Z/(4m\mathbb Z)\to H_1(Z_{\widetilde K}(2m))\cong\mathbb Z/(2m(2m+1)\mathbb Z)\]
be the transfer homomorphism, then $[\widetilde T']=\tau_*([T'])$ and $[\widetilde K']=\tau_*([K'])$. Since $\gcd(4m,2m+1)=1$, the order of any element in the image of $\tau_*$ is a divisor of $2m$. It follows from (1) that $[\widetilde K']=\pm[\widetilde T']$.
\end{proof}

\begin{lem}\label{lem:TDualSimple}
$\widetilde T'$ is a simple knot in the lens space $Z_{\widetilde K}(2m)$.
\end{lem}
\begin{proof}[{\bf Proof}]
The knot $T=T(2m+1,2)$ is also a pretzel knot $P(2,-1,2m+3)$. Using this pretzel diagram, it is easy to find a  once-punctured Klein bottle $P_T$ in the complement of $T$, such that the boundary slope of $P_T$ is $4m$. Moreover, the complement of a neighborhood of $P_T$ is a genus-$2$ handlebody with respect to which $T$ is primitive. So the same argument as in the proof of Lemma~\ref{lem:1Bridge} shows that $\widetilde T'$ is an $1$--bridge knot in the lens space $Z_{\widetilde K}(2m)$. Then the same argument as in the proof of Corollary~\ref{cor:DualSimple} shows that $\widetilde T'$ is simple.
\end{proof}

Now we are ready to prove Theorem~\ref{thm:DiSurg}.

\begin{proof}[{\bf Proof of Theorem~\ref{thm:DiSurg}}]
As mentioned in the first paragraph of this section, our theorem holds when $K$ is non-hyperbolic. So we assume $K$ is hyperbolic. Proposition~\ref{prop:GenusBound} and Corollary~\ref{punctured Klein bottle} imply that $\e=+$ and $n=m$. Corollary~\ref{cor:DualSimple} and Lemma~\ref{lem:TDualSimple} say that both $\widetilde K'$ and $\widetilde T'$ are simple knots in the lens space $Z_{\widetilde K}(2m)$. Now Lemma~\ref{lem:SameHomo} implies that $\widetilde K'$ and $\widetilde T'$ are isotopic up to orientation reversal. But this is impossible since the complement of $\widetilde K'$ is hyperbolic, while the complement of $\widetilde T'$ is Seifert fibered.
\end{proof}

\end{document}